\documentclass[11pt]{amsart}
\usepackage{amscd, amssymb, pgf, tikz, hyperref}

\newcommand{\field}[1]{\mathbb{#1}}
\newcommand{\CC}{\field{C}}
\newcommand{\NN}{\field{N}}

\newcommand{\TT}{\field{T}}
\newcommand{\ZZ}{\field{Z}}
\newcommand{\RR}{\field{R}}

\newcommand{\Bb}{\mathcal{B}}
\newcommand{\Cc}{\mathcal{C}}

\newcommand{\Hh}{\mathcal{H}}
\newcommand{\Kk}{\mathcal{K}}
\newcommand{\Ll}{\mathcal{L}}
\newcommand{\Mm}{\mathcal{M}}

\newcommand{\Oo}{\mathcal{O}}

\newcommand{\ol}{\overline{\lambda}}

\newcommand{\la}{\langle}
\newcommand{\ra}{\rangle}

\newcommand{\ds}{\displaystyle}

\newcommand{\tm}{\tilde{\mu}}
\newcommand{\Ind}{\operatorname{Ind}}
\newcommand{\go}{G^{(0)}}

\newcommand{\fS}{\frak{F}}

\newtheorem{thm}{Theorem}[section]
\newtheorem{cor}[thm]{Corollary}
\newtheorem{lem}[thm]{Lemma}
\newtheorem{prop}[thm]{Proposition}

\theoremstyle{definition}
\newtheorem{dfn}[thm]{Definition}

\theoremstyle{remark}
\newtheorem{rmk}[thm]{Remark}

\newtheorem{example}[thm]{Example}

\newtheorem*{examples*}{Examples}

\numberwithin{equation}{subsection}
\numberwithin{equation}{subsection}

\title{ Groupoid actions and  Koopman representations}

\author{Valentin Deaconu}
\address{Valentin Deaconu \\ Department of Mathematics (0084)\\ University
of Nevada\\ Reno NV 89557-0084\\ USA} \email{vdeaconu@unr.edu}

\author{Marius Ionescu}
\address{ Marius Ionescu\\ United States Naval Academy\\ Annapolis\\ MD 21402-5002\\ USA}
\email{ionescu@usna.edu}
\thanks{This work was partially supported by a AMS-Simons Research Enhancement Grant for PUI Faculty}


\begin{document}
\begin{abstract}
We study the $C^*$-algebra $C^*(\kappa)$ generated by the Koopman representation $\kappa=\kappa^\mu$ of a locally compact groupoid $G$ acting on a measure space $(X,\mu)$, where $\mu$ is  quasi-invariant for the action. We interpret $\kappa$ as an induced representation and we prove that if  the groupoid $G\ltimes X$ is amenable, then $\kappa$ is weakly contained in the regular representation $\rho=\rho^\mu$ associated to $\mu$, so we have a surjective homomorphism $C^*_r(G)\to C^*(\kappa)$. We consider the particular case of Renault-Deaconu groupoids $G= G(X,T)$ acting on their unit space $X$ and show that in some cases $C^*(\kappa)\cong C^*(G)$.
 
\end{abstract}
\maketitle
\section{introduction}

\bigskip

The concept of a group action on a space was generalized to a groupoid
action and it has applications to dynamical systems, representation
theory and operator algebras. If groups can roughly be described  as
the set of symmetries of certain objects, then groupoids can be
thought as the set of symmetries of fibered objects. 

A  unitary  representation of a locally compact groupoid $G$
endowed with a Haar system    is a triple
$(\mu,\go*\Hh,\hat{L})$ consisting of a quasi-invariant measure
$\mu$ on the unit space $\go$ of $G$, a Borel Hilbert bundle $\go*\Hh$ over
$\go$, and a Borel homomorphism $\hat{L}:G\to
\text{Iso}(\go*\Hh)$ such that $\hat{L}(g)=(r(g),L_g,s(g))$ and
$L_g:\Hh(s(g))\to \Hh(r(g))$ is a Hilbert space isomorphism (cf. (\cite[Definition II.1.6]{R80}; see also \cite[Definition
7.7]{W},\cite{R87,R91})). A Koopman representation of $G$ is a unitary
representation of $G$ determined by a pair $(X,\mu)$ consisting of a separable
locally compact space $X$ on which $G$ acts leaving the measure
$\mu$ quasi-invariant. That is, $X$ is fibered over $\go$ by a continuous open
surjection $\omega:X\to\go$, and $\mu$ admits a disintegration
$d\mu(\cdot)=\int_{X_u}d\mu_u(\cdot)d\tm(u)$ where each $\mu_u$ is a
probability measure supported on $X_u:=\omega^{-1}(u)$ and $\tm=\omega_*(\mu)$ is a
probability measure on $\go$ that is quasi-invariant in the usual sense. The
Hilbert bundle $\Hh$ for the Koopman representation determined by $(X,\mu)$ is
$\{L^2(X_u,\mu_u)\}_{u\in\go}$ and the representation $L$, denoted here by
$\kappa^\mu$, is given by 
\[
  \kappa^\mu_g:L^2(X_{s(g)},\mu_{s(g)})\to L^2(X_{r(g)},\mu_{r(g)}),
\]

\[
  \kappa_g^\mu\xi(x):=D(g^{-1},x)^{\frac12}\xi(g^{-1}x),
\]
 where $D(g,\cdot)$ is the Radon-Nikodym derivative
$d(g\mu_{s(g)}/d\mu_{r(g)})$ (see Section \ref{sec:induc-repr-from} for
  details). It is our main goal to study the $C^*$-algebra $C^*(\kappa^\mu)$ generated by the Koopman representation.

We begin  by fixing some notation associated with a locally compact
 Hausdorff groupoid $G$ with a Haar sytem. We recall the definition of
a groupoid action $G\curvearrowright X$ on a locally compact space $X$
fibered over the unit space of $G$ and of additional concepts like
orbits, stabilizers and transitive actions. We illustrate with several
examples of actions, including the cases $X=G^{(0)}, X=G$, $X=G/H$ for
$H$ a closed subgroupoid and $\ds X=\bigcup_{x\in G^{(0)}}G(x,S)$,
where $G(x,S)$ is the Cayley graph for a generating set $S$. We also
review the definition of the action groupoid $G\ltimes X$ and of the
concepts of groupoid fibration and groupoid covering.  

We continue with quasi-invariant measures on $X$ for
$G\curvearrowright X$ and relate them to measures for the action
groupoid $G\ltimes X$.  We recall some facts about groupoid
representations, induced representations  and amenability.  The
Koopman representation $\kappa^\mu: G\to \Bb(L^2(X,\mu))$ associated
to a quasi-invariant probability measure $\mu$ on $X$ can be
understood as the induced representation of the trivial representation
$i_{G\ltimes X}$. Induced representations in the case $G$ is a Borel
transformation group groupoid already appeared in Definition 3.5 of
\cite{Ra76}. When $G$ acts on itself by left multiplication, the
Koopman 
representation is just the left regular representation.  

We prefer to work with unitary representations of groupoids which
appear in a natural way in our context rather than with the integrated
forms at the level of $C^*$-algebras. Most of our results could
be recast in terms of Hilbert modules \`a la Rieffel.  Holkar has
already shown in \cite{H} that Rieffel's construction of induced
representations is valid for topological groupoid correspondences. We
believe that Renault's perspective from \cite{R14} of inducing unitary
representation at the groupoid level is better suited for examples and
to illustrate how one can recover the classical definitions and
results going back to Mackey's work on induced representations of groups.

We define $C^*(\kappa^\mu)$ to be the closure of $\kappa^\mu(C_c(G))$
in $ \Bb(L^2(X,\mu))$ and we try to relate it to $C^*_r(G)$. We prove
that if the action groupoid $G\ltimes X$ is $\sigma$-compact and
amenable and the measure $\mu$ has full support, then the Koopman
representation is weakly contained in the left regular representation
associated to $\mu$, so we have a surjective homomorphism $C^*_r(G)\to
C^*(\kappa^\mu)$. In some cases (see the examples involving graph $C^*$-algebras in section 6), this is an isomorphism.

In the case when the Renault-Deaconu groupoid $G(X,T)$ associated to a
local homeomorphism $T:X\to X$ acts on a space  $Y$, it is known that
the action groupoid is isomorphic to another Renault-Deaconu groupoid,
see \cite{IK}.  The form of quasi-invariant measures for $G(X,T)$ with
given Radon-Nikodym derivative is studied in several papers, like
\cite{KR,IK, R03}. We illustrate the theory with several examples in the last section of the paper.

\subsection*{Acknowledgments}The authors would like to thank Marcelo Laca whose suggestions led to
an improvement of our results, recovering the ideal structure of a graph $C^*$-algebra from particular Koopman representations.

\bigskip

\section{Groupoid actions }

\bigskip

A groupoid $G$ is a  small category with inverses. We will use $s$ and
$r$ for the source and range maps $s,r:G \to G^{(0)}$, where $G^{(0)}$
is the unit space. We always assume that $G$ has  a locally compact
Hausdorff topology compatible with the algebraic structure. To
construct $C^*$-algebras from a groupoid $G$, we will assume that $G$
is  second countable with a Haar system. An \'etale groupoid is a
topological groupoid where the range map $r$ (and necessarily the 
source map  $s$) is a local homeomorphism. The unit space $G^{(0)}$ of an \'etale groupoid is always an open subset of $G$ and a Haar system is given by the counting measures. 

The set of composable pairs is denoted by $G^{(2)}$. Let $G_u$ be the set of $g\in G$ with $s(g)=u$, let $G^v$ be the set of $g\in G$ with $r(g)=v$, and let $G_u^v=G_u\cap G^v$. Two units $x,y \in G^{(0)}$ belong to the same $G$-orbit if there exists $g \in G$ such that $s(g) = x$ and $r(g) = y$. When every $G$-orbit is dense in $G^{(0)}$,  the groupoid $G$ is called minimal. 

The isotropy group of a unit $x\in G^{(0)}$ is the group \[G_x^x :=\{g\in G\; | \; s(g)=r(g)=x\},\] and the isotropy bundle is
\[G' :=\{g\in G\; | \; s(g)=r(g)\}= \bigcup_{x\in G^{(0)}} G_x^x.\]
A groupoid $G$ is said to be principal if all isotropy groups are trivial, or equivalently, $G' = G^{(0)}$. 

\begin{dfn} Let $G$ be a topological groupoid. A bisection is a subset $U\subseteq  G$ such that $s$ and $r$ are both injective when restricted to $U$.
\end{dfn}
An open bisection $U$ determines a homeomorphism $\pi_U=(r|_U)\circ(s|_U)^{-1}:s(U)\to r(U), \pi_U(x)=r(s^{-1}(x))$. An \'etale groupoid has sufficiently many open bisections which generate its topology.

\begin{example}\label{RD}

Let $X$ be a locally compact Hausdorff space and let $T:X\to X$ be a local homeomorphism. The Renault-Deaconu groupoid associated to $T$ is
\[G(X,T)=\{(x,m-n,y)\in X\times \mathbb Z\times X:T^m(x)=T^n(y)\}\]
with operations
\[(x,k,y)(y,\ell, z)=(x, k+\ell, z),\; (x,k,y)^{-1}=(y,-k,x).\]
We identify the unit space of $G(X,T)$ with $X$ via the map $(x,0,x)\mapsto x$. The range and source maps are then
\[r(x,k,y)=x,\; s(x,k,y)=y.\]
A basis for the topology consists of sets of the form
\[Z(U,m,n,V)=\{(x,m-n,y): T^m(x)=T^n(y), x\in U, y\in V\},\]
where $U,V$ are open subsets of $X$ such that $T^m|_U$ and $T^n|_V$ are one-to-one and $T^m(U)=T^n(V)$. These are bisections for $G(X,T)$, and with this topology, $G(X,T)$ becomes an \'etale  groupoid.
\end{example}

\bigskip

 We now recall the definition of a groupoid action on a space given in
 \cite[Definition 2.1]{W} or \cite[Definition 4.1]{AD23}:

\begin{dfn}\label{sa}  A topological groupoid $G$ is said to act
(on the left) on a locally compact space $X$, if there are given 
a continuous  surjection $\omega : X \rightarrow G^{(0)}$,  called the anchor or moment map,
and a continuous map
\[G\ast X \rightarrow X, \quad\text{write}\quad (g , x)\mapsto
g \cdot x=gx,\]
where
\[G \ast X = \{(g , x)\in G \times X \mid s(g) = \omega (x)\},\]
that satisfy

\medskip

i) $\omega (g \cdot x) =r(g)$ for all $(g , x) \in G \ast X,$

\medskip

ii) $(g _2, x) \in G \ast X,\,\, (g_1, g_2)
\in G ^{(2)}$ implies $(g _1g _2, x),
(g _1, g _2\cdot x) \in G * X$ and
\[g _1\cdot(g _2\cdot x) = (g _1g _2)
\cdot x,\]

\medskip

iii) $\omega (x)\cdot x = x$ for all $x\in X$.
 
\noindent We denote by $X_u$  the fiber $\omega^{-1}(u)$ over $u\in G^{(0)}$.
\end{dfn}
We should mention that in \cite[Section 2]{MRW1} the authors required that the
anchor map is open as well. 

The action of $G$ on $X$  is called transitive if given $x,y\in X$, there is $g\in G$ with $g\cdot x=y$ and is free if $g\cdot x=x$ for some $x$ implies $g=\omega(x)\in G^{(0)}$.

The set of fixed points in $X$ is defined as \[X^G=\{x\in X: g\cdot x=x\;\text{for all}\; g\in G_{\omega(x)}^{\omega(x)}\}.\] 
If $G$ has trivial isotropy, then $X^G=X$.

The orbit of $x\in X$ is \[Gx=\{g\cdot x: g\in G, \;
  s(g)=\omega(x)\}.\] The set of orbits is denoted by $G\backslash X$
and has the quotient topology. The action of $G$ on $X$ is called
minimal if every orbit $Gx$ is dense in $X$. For a transitive action there is a single orbit.

For $x\in X$, its  stabilizer group is
\[G(x)=\{g\in G: g\cdot x=x\},\] which is a subgroup of  $G_u^u$ for
$u=\omega(x)$. 
\begin{rmk}
Note that if the action of $G$ on $X$ is transitive, then $G(x)\cong G(y)$ for $x,y\in X$. Indeed, if $h\cdot x=y$, then $g\mapsto hgh^{-1}$ is an isomorphism $G(x)\to G(y)$.
For a transitive action, we may consider
$\ds H=\bigcup_{u\in G^{(0)}}G(x)$ (here we pick an $x$ in each fiber $X_u=\omega^{-1}(u)$), which is a subgroupoid of the isotropy $\ds G'=\bigcup_{u\in G^{(0)}}G_u^u$. Then $H$ is a normal subgroupoid of $G$, in the sense that $ghg^{-1}\in H$ for $h\in H$ and $g\in G$ with $s(g)=r(h)$. Indeed, for $h\in G(x)$, we have $ghg^{-1}\in G(g\cdot x)$. The quotient groupoid $G/H$ made of left cosets  can be identified with $X$ using $gG(x)\mapsto g\cdot x$.
\end{rmk}

\begin{example}\label{ex:actionunit} A groupoid $G$ with open source and range maps acts on
  its unit space $G^{(0)}$ by $g\cdot s(g)=r(g)$. In this case,
  $\omega=id$. The groupoid is called transitive if this action is
  transitive. Notice that $g\cdot u=u$ for all $g\in G_u^u$, in
  particular $G(u)=G_u^u$ and $(G^{(0)})^G=G^{(0)}$.  A transitive groupoid with
  discrete unit space is of the form $G^{(0)}\times K\times G^{(0)}$ with usual operations,
  where $K$ is a copy of the isotropy group.  
\end{example}

\begin{example}\label{homogeneous} A groupoid $G$ acts on itself by left multiplication with $\omega(g)=r(g)$.  More general, if $G$ is a groupoid and $H$ is a closed  subgroupoid, then $G$ acts on the set of left cosets $G/H$ by left multiplication. Here $\omega(gH)=r(g)$. Note that this action is not necessarily transitive, since given $g_1H, g_2H\in G/H$, the element $g_2g_1^{-1}$ is defined only for $s(g_1)=s(g_2)$. 
\end{example}


\begin{rmk}If $G$ acts on $X$, the fibered product \[G \ast X = \{(g , x)\in G \times X \mid s(g) = \omega (x)\}\] has a natural
structure of  groupoid, called the semi-direct product or
action groupoid and is denoted  by $G \ltimes X$, where
\[(G \ltimes X)^{(2)} = \{ ((g_1, x_1),(g _2, 
x_2)) \mid \,\,x_1 = g _2\cdot x_2\},\]
with operations
\[(g _1, g_2\cdot x_2)(g _2,  x_2) = 
(g _1g _2, x_2), \;\;
(g,x)^{-1} = (g ^{-1}, g \cdot x).\]
The source and range maps of $G \ltimes X$ are
 \[s(g ,x) = (s(g),x)=
(\omega( x), x),\quad r(g ,x) =  (r(g), g\cdot x)=(\omega(g\cdot  x),g\cdot x),\]
and the unit space $(G\ltimes X)^{(0)}$ may be identified with $X$ via the map \[i:X\to G\ltimes X,\;\; i(x)= (\omega(x),x).\] 
Note that the source and range maps defined above are open  even if
the anchor map $\omega$ is not assumed to be open (see \cite[page 10]{AD23}).

Recall from \cite{BM,  DKR} that a {\em groupoid fibration} is a surjective open morphism
of locally compact groupoids $\pi :G\rightarrow H$ with the property
that for all $h\in H$ and $x\in G^{(0)}$ with $\pi(x)=s(h)$ there is $g\in
G$ with $s(g)=x$ and $\pi(g)=h$. 
If $g$ is unique for any such $h$ and $x$, then $\pi$ is called 
 a {\em groupoid covering}. Note that for a groupoid covering we have
 $\pi^{-1}(H^{(0)})=G^{(0)}$.

For $G$ acting on $X$, the projection map\[\pi:G\ltimes X\to G, \;\; \pi(g,x)=g\] is a covering of groupoids. Conversely, given a covering of groupoids $\pi:G\to H$, there is an action of $H$ on $X=G^{(0)}$ with $\omega=\pi|_{G^{(0)}}:G^{(0)}\to H^{(0)}$ and $G\cong H\ltimes X$. The action is defined by $h\cdot x=r(g)$, where $g\in G$ is unique with $\pi(g)=h$ and the isomorphism is given by $g\mapsto (\pi(g), s(g))$.
\end{rmk}
Note that for $G$ acting on $X=G^{(0)}$ by $g\cdot s(g)=r(g)$, we get $G \ltimes X\cong G$.

\begin{example}\label{skew}
Consider $E=(E^0,E^1,r,s)$ a topological graph and $G$ a topological groupoid. Recall that $E^0,E^1$ are locally compact Hausdorff spaces, $r:E^1\to E^0$ is continuous and $s:E^1\to E^0$ is a local homeomorphism. Let $c:E^0\cup E^1\to G$ be a continuous function such that $c(E^0)\subset G^{(0)}$, $c(s(e))=s(c(e)), \;\; c(r(e))=r(c(e))$ and such that $(c(e_1),c(e_2))\in G^{(2)}$ for $e_1e_2\in E^2$. The map $c$ is called a cocycle and it can be extended to finite paths by $c(e_1e_2\cdots e_k)=c(e_1)c(e_2)\cdots c(e_k)$.

The skew-product graph $E\times_cG$ has vertices \[E^0\times_c G=\{(v,g): c(v)=s(g)),\] edges
\[E^1\times_c G=\{(e,g): (g,c(e))\in G^{(2)}\}\] and incidence maps \[\tilde{r}(e,g)=(r(e), gc(e)),\;\; \tilde{s}(e,g)=(s(e), g).\] Then $(E^0\times_cG, E^1\times_cG, \tilde{r}, \tilde{s})$ becomes a topological graph since $\tilde{s}$ is a local homeomorphism and $\tilde{r}$ is continuous. Moreover, $G$ acts freely on $E^0\times_cG$ by $h\cdot(v,g)=(v,hg)$, where $\omega:E^0\times_cG\to G^{(0)}, \omega(v,g)=r(g)$. Similarly, $G$ acts freely on $E^1\times_cG$ by $h\cdot(e,g)=(e,hg)$ with $\omega:E^1\times_cG\to G^{(0)}, \omega(e,g)=r(g)$. The action commutes with the incidence maps and the quotient graph is isomorphic to $E$.
\end{example}

\begin{example}\label{Cayley}
Let $G$ be a topological groupoid. We say that $Y\subset G^{(0)}$ is a topological transversal if $Y$ contains an open transversal (recall that a transversal intersects every orbit). A compact generating pair $(S,Y)$ of $G$ is made of a compact subset $S\subset G$ and a compact topological transversal $Y$ such that for every $g\in G|_Y=\{g\in G: s(g),r(g)\in Y\}$ there exists $n$ such that $\ds \bigcup_{0\le k\le n}(S\cup S^{-1})^k$ is a neighborhood of $g$ in $G|_Y$. Here, for a subset $A\subset G$, $A^k$ is the set of all products $a_1\cdot a_2\cdots a_k$ where $a_i\in A$.

If $(S,Y)$ is a compact generating pair for $G$ and $x\in Y$, the Cayley graph $G(x,S)$ is the directed graph with vertex set $G_Y^x=\{g\in G: s(g)\in Y, r(g)=x\}$ such that there is an edge from $g_1$ to $g_2$ whenever there is $h\in S$ with $g_2=g_1h$. If $G^{(0)}$ is compact, then the groupoid $G$ with generating set $(S,G^{(0)})$ acts freely on the union of Cayley graphs $\ds\bigcup_{x\in G^{(0)}}G(x,S)$ by left multiplication.

In particular, if $\sigma:\TT\to \TT, \sigma(z)=z^d$ for $d\ge 2$ and $G=G(\TT,\sigma)$ is  the groupoid of germs of the pseudogroup generated by $\sigma$ (see section 2 in \cite{R00}), then we can take the generating set $S$ to be a finite set of germs of  maps $\sigma^{-1}:\sigma(U)\to U$, where $U\subseteq \TT$ is an open set such that $\sigma:U\to \sigma(U)$ is a homeomorphism. Then the Cayley graphs $G(z,S)$ are regular trees of degree $d+1$ and the groupoid $G$ acts on their union. For more on Cayley graphs of groupoids, see \cite{N}.

\end{example}
\bigskip

\section{Quasi-invariant measures and representations}

\bigskip

Let $G$ be a locally compact groupoid with left Haar system
$\{\lambda^u\}_{u\in G^{(0)}}$ and let $\mu$ be a measure on
$G^{(0)}$. The measure $\nu=\mu\circ\lambda$ on $G$ induced by $\mu$ is defined
via
\[
  \int_G f(g)\,d\nu(g)=\int_{G^{(0)}}\int_{G^u}f(g)\,d\lambda^u(g)d\mu(u)
\]
for all $f\in C_c(G)$. Let $\nu^{-1}$ be the push-forward of $\nu$ under
the inverse map. 

\begin{dfn}
  A measure $\mu$ on $G^{(0)}$ is called quasi-invariant
  (\cite[Definition I.3.2]{R80}) if its induced measure $\nu$ is
  equivalent to its inverse $\nu^{-1}$,  i.e. they have the same nullsets (we write $\nu\sim \nu^{-1}$ in this case).
  
  \end{dfn}
  
 \begin{rmk}
For an \'etale groupoid $G$, a Radon measure $\mu$ on $G^{(0)}$ is quasi-invariant if for all open bisections $U$, the measures $({\pi_U}_*\mu) |_{s(U)}$ and $\mu |_{r(U)}$ are equivalent. Here $\pi_U:s(U)\to r(U),\; \pi_U(x)=r(s^{-1}(x))$ and $({\pi_U}_*\mu)(B)=\mu(\pi_U^{-1}(B))$ for $B\subset G^{(0)}$ a Borel set.

\end{rmk}

\begin{rmk}
  Recall from \cite[Proposition I.3.3]{R80} that if $\Delta_\mu:G\to(0,\infty)$
  is a Radon-Nikodym derivative such that
  \[
   \int_{G^{(0)}}\int_{G^u}f(g)\,d\lambda^u(g)d\mu(u)=\int_{G^{(0)}}\int_{G_u}f(g)\Delta_\mu(g)\,d\lambda_u(g)d\mu(u), 
  \]
  where $\lambda_u$ is the push forward of $\lambda^u$ under the
  inversion map for all $u\in
  G^{(0)}$, then $\Delta_\mu$ is
  a cocycle a.e. Moreover, \cite[Theorem 3.2]{Ra82} implies that one
  can choose $\Delta_\mu$ to be a strict cocycle: $\Delta_\mu(gh)=\Delta_\mu(g)\Delta_\mu(h)$ for all
  $(g,h)\in G^{(2)}$ and $\Delta_\mu(g^{-1})=\Delta_\mu(g)^{-1}$ for all $g\in G$.
\end{rmk}

\bigskip

\bigskip

\begin{example}
  Let $\sigma:\TT\to\TT$ be $\sigma(z)=z^d$ for $d\ge 2$ an
  integer. Let $G(\TT,\sigma)$ be the associated  Renault-Deaconu groupoid, isomorphic to the groupoid of germs of the pseudogroup generated by $\sigma$. Then
  the Haar measure $\mu$ on $\TT$ is quasi-invariant and
  $\Delta_\mu(z,k-l,w)=(1/d)^{k-l}$ for all $(z,k-l,w)\in G(\TT,\sigma)$.
\end{example}

We assume now that the topological groupoid $G$ acts on the space $X$ via $\omega:X\to G^{(0)}$. 
Recall that if $\mu '$ is a finite nontrivial measure on $X$ then there is a
probability measure $\mu$ on $X$ such that $\mu '\sim \mu$, i.e. they have the same nullsets. 
\begin{dfn}
 Suppose that $\mu$ is a Radon probability  measure on $X$ and let
  $\tm:=\omega_*(\mu)$ on $G^{(0)}$. That is, $\tm(B)=\mu(\omega^{-1}(B))$ for all
  Borel sets $B\subset G^{(0)}$. 
  
  A decomposition of $\mu$ relative to
  $\omega$ is a family of measures $\{\mu_u\}_{u\in G^{(0)}}$ such that
  \begin{enumerate}
  \item $\operatorname{supp}\mu_u=X_u=\omega^{-1}(u)$ for all $u\in
    G^{(0)}$ and
  \item for all $f\in C_c(X)$, the map $\ds u\mapsto
    \int_{X_u}f(x)\,d\mu_u(x)$ belongs to $C_c(G^{(0)})$ and
    \[
      \int_X f(x)\,d\mu(x)=\int_{G^{(0)}}\int_{X_u}f(x)\,d\mu_u(x)\,d\tm(u).
    \]
  \end{enumerate}
\end{dfn}

\begin{dfn}\label{def:qi}
Let $\mu$ be a  Radon probability measure on $X$. We say that $\mu$ is
$G$-quasi-invariant for the action of $G$ on $X$ if it admits a decomposition $\{\mu_u\}$ relative
to $\omega$ such that both of the following conditions hold:
\begin{enumerate}
\item For all $g\in G$, the measure $g\mu_{s(g)}$ is equivalent with
  $\mu_{r(g)}$, where $g\mu_{s(g)}(B):=\mu_{s(g)}(g^{-1}B)$ for any
  Borel set $B\subseteq X_{r(g)}$; and
\item the measure $\tm=\omega_*(\mu)$ on $G^{(0)}$ is quasi-invariant for the
  groupoid $G$. 
\end{enumerate}

\end{dfn}

Both of the two conditions in the definition are needed as the
following examples show.
\begin{example}
  \begin{enumerate}
  \item Assume that $G$ is a locally compact \emph{group} acting on a
    locally compact Hausdorff space. Therefore $G^{(0)}=\{e\}$ and
    $\omega(x)=e$ for all $x\in X$. Let $\mu$ be a probability measure
    on $X$. Then $\tm=\delta_e$, the point mass at $e$, and $\mu_e=\mu$. Hence $\mu$ is
    $G$-quasi-invariant in the sense of Definition \ref{def:qi} if and only
    if it is $G$-quasi-invariant in the classical sense: the measure $g\mu$
    is equivalent to $\mu$ for all $g\in G$, where
    $g\mu(B)=\mu(g^{-1}B)$.
  \item Assume that $G$ is a locally compact Hausdorff groupoid that
    acts on its unit space $X=G^{(0)}$ as in Example
    \ref{ex:actionunit}. Thus $\omega(x)=x$ for all $x\in
    G^{(0)}$, $\tm=\mu$, $X_x=\{x\}$, and, hence,
    $\mu_x=\delta_x$. Therefore a measure $\mu$ on $G^{(0)}$ is
    $G$-quasi-invariant in the sense of Definition \ref{def:qi} if and
    only if $\mu$ is a quasi-invariant measure for $G$ in the usual sense.
  \end{enumerate}
\end{example}

Renault defined in \cite[Definition 2.2]{R14} a $G$-quasi-invariant
measure to be a  quasi-invariant measure for the action groupoid. The
following theorem proves that the two definitions are equivalent. For
the case of Borel groupoids, this result is Corollary 5.3.11 in \cite
{AR} and a similar result appears in \cite[Proposition 3.1]{R14}. 

\begin{thm}\label{thm:quasi-invariant} If the groupoid $G$ acts on $X$, then a
  measure $\mu$ on $X$ is $G$-quasi-invariant iff $\mu$ is
  quasi-invariant for the action groupoid $G\ltimes X$ with unit space
  $X$. 
\end{thm}
\begin{rmk}{\label{rm:haarsystem}}
  Recall that if $\lambda=\{\lambda^u\}_{u\in G^{(0)}}$ is a Haar system for
  $G$ then $\ol=\{\ol^x\}_{x\in X}$ defined via
  \[
    \int_{(G\ltimes X)^{x}} f(g,y)\,d\ol^x(g,y):=\int_{G^{\omega(x)}}f(g,g^{-1}\cdot x)\,d\lambda^{\omega(x)}(g)
  \]
 for all $f\in C_c(G\ltimes X)$ and $x\in X$ is a Haar system on $G\ltimes X$. We will use this Haar system for
  the action groupoid, see Ex. 2.1.7 on page 37 in \cite{W} and
  \cite[page 10]{AD23}.
\end{rmk}

\begin{proof}
  Assume that $\mu$ is a $G$-quasi-invariant measure on $X$. For $g\in
  G$, let $D(g,\cdot)$ be the Radon-Nikodym derivative
  $d(g\mu_{s(g)})/d\mu_{r(g)}$.  Note that since $\mu$ is a
  $G$-quasi-invariant measure and $\{\lambda^u\}$ is a Haar system on
  $G$,  $(G\ltimes X, \mu\circ \ol)$ is a measured groupoid (see
  \cite{Ra71},\cite{Ra82}). Therefore, using virtually the same
  arguments as in
  the proof of \cite[Corollary D.34]{Wcp}, we can choose $D$ to be Borel
   and  $D(g_1g_2,x)=D(g_1,g_2\cdot x)D(g_2, x)$ for all
   $(g_1,g_2)\in G^{(2)}$ and $\mu$-almost all $x$. Let $\Delta_{\tm}$ be the modular function
  associated with $\tm $ and set
  $\Delta_{\mu}(g,x):=D(g,x)\Delta_{\tm}(g)$ for all $(g,x)\in G\ltimes
  X$. Let $\ol$ be the Haar system on $G\ltimes X$ and let
  $\overline{\nu}:=\mu\circ \ol$. We prove that $\overline{\nu}\sim
  \overline{\nu}^{-1}$ and that a Radon-Nikodym derivative is given by
  $\Delta_{\mu}(g,x)$. Let $f\in C_c(G\ltimes X)$. We have
  \begin{align*}
    \int_{G\ltimes X}
    f(g,x)\,d\overline{\nu}(g,x)&=\int_X\int_{G^{\omega(x)}}f(g,g^{-1}\cdot
                                  x)\,d\lambda^{\omega(x)}(g)\,d\mu(x)\\
    &=\int_{G^{(0)}}\int_{X_u}\int_{G^u}f(g,g^{-1}\cdot
      x)\,d\lambda^u(g)\,d\mu_u(x)\,d\tm (u)\\
      \intertext{which, by Fubini's theorem,}
    &=\int_{G^{(0)}}\int_{G^u}\int_{X_u}f(g,g^{-1}\cdot
      x)\,d\mu_u(x)\,d\lambda^u(g)\,d\tm (u)\\
      \intertext{which, since $g\mu_{s(g)}\sim \mu_{r(g)}$,} 
                                &=\int_{G^{(0)}}\int_{G^u}\int_{X_{s(g)}}f(g,x)D(g,x)\,d\mu_{s(g)}(x)\,d\lambda^u(g)\,d\tm (u)\\
                                \intertext{which, since $\tm $ is
                              quasi-invariant for $G$, }
                                &=\int_{G^{(0)}}\int_{G_{u}}\int_{X_u}f(g,x)D(g,x)\,d\mu_u(x)\Delta_{\tm}(g)\,d\lambda_u(g)\,d\tm (u)\\
                                \intertext{which, using Fubini's theorem again,
                                }
                                &=\int_{G^{(0)}}\int_{X_u}\int_{G_u}f(g,x)D(g,x)\Delta_{\tm}(g)\,d\lambda_u(g)\,\mu_u(x)\,d\tm (u)\\
    &=\int_X\int_{G_u}f(g,x)D(g,x)\Delta_{\tm}(g)\,d\lambda_u(g)\,d\mu(x)\\
    &=\int_{G\ltimes X} f(g,x)\Delta_{\mu}(g,x)\,d\overline{\nu}^{-1}(g,x).
  \end{align*}
  Thus $\mu$ is quasi-invariant for $G\ltimes X$. 

  Assume now that $\mu$ is a quasi-invariant measure on $X$ for the action groupoid
$G\ltimes X$ and let $\Delta_{\mu}(g,x)$ be the associated Radon-Nikodym derivative.
Using \cite[Theorem I.5]{Wcp} we disintegrate $\mu$ with respect
to $\tm=\omega_*(\mu)$,
\[
  \int_X f(x)\,d\mu(x)=\int_{G^{(0)}} \int_{X_u}
  f(x)\,d\mu_u(x)\,d\tm (u),
\]
where $\{\mu_u\}$ is a family of Radon probability measures with
$\operatorname{supp}\mu_u\subseteq X_u$ and $f\in C_c(X)$. We prove first that
$\tm $ is a quasi-invariant measure for $G$. Let $f\in C_c(G)$. Then
\begin{align*}
  \tm \circ\lambda(f)&=\int_{G^{(0)}}\int_{G^{u}}f(g)\,d\lambda^u(g)\,d\tm (u)=\int_X\int_{G^{\omega(x)}}f(g)\,d\lambda^{\omega(x)}(g)\,d\mu(x)\\
  \intertext{which  by the quasi-invariance of $\mu$  }
  &=\int_X
  \int_{G_{\omega(x)}}f(g)\Delta_{\mu}(g,x)\,d\lambda_{\omega(x)}(g)\,d\mu(x)\\
  &=\int_{G^{(0)}}\int_{X_u}\int_{G_u}f(g)\Delta_{\mu}(g,x)\,d\lambda_u(x)\,d\mu_u(x)\,d\tm (u)\\
  \intertext{ which by Fubini's theorem}
  &=\int_{G^{(0)}}\int_{G_u}f(g)\left(\int_{X_u}\Delta_{\mu}(g,x)\,d\mu_u(x)\right)\,d\lambda_u(g)\,d\tm (u)\\
  \intertext{which, by defining
    $\ds\Delta_{\tm}(g):=\int_{X_{s(g)}}\Delta_{\mu}(g,x)\,d\mu_{s(g)}(x)$,
  }
  &=\int_{G^{(0)}}\int_{G_u}f(g)\Delta_{\tm}(g)\,d\lambda_u(g)\,d\tm (u)=\int_Gf(g)\Delta_{\tm}(g)\,d(\tm \circ
  \lambda)^{-1}(g).
\end{align*}
Therefore $\tm \circ \lambda\sim (\tm \circ \lambda)^{-1}$ and,
thus, $\tm $ is quasi-invariant for $G$.

Let $g\in G$. Then the set $U:=\{(g,x)\,:\,x\in X_{s(g)}\}$
is a  measurable bisection with respect to
$\overline{\nu}=\mu\circ \ol$. Note that
$s(U)=X_{s(g)}$ and $r(U)=X_{r(g)}$. Using the fact
that $\{\lambda^u\}$ is a Haar system for $G$, one can check that
$\overline{\nu}$ is quasi-invariant under $U$ in the sense
of \cite[Definition I.3.18 i)]{R80}. Since $\mu$ is quasi-invariant for
$G\ltimes X$, Proposition I.3.20 of \cite{R80} implies that $\mu$ is quasi-invariant
under $U$ in the sense of \cite[Definition I.3.18 ii)]{R80}. Thus
$g\mu_{s(g)}\sim \mu_{r(g)}$. 
\end{proof}

 \begin{dfn}
Given a groupoid $G$, if  $\mu$ is any Radon measure on $G^{(0)}$, then 
the regular representation on $\mu$, denoted  $\Ind\mu$, 
acts on $L^2(G,\nu^{-1})$ via
\begin{equation}
  \label{eq:ind_mu}
  \Ind\mu(f)(\xi)(g)=\int_G f(h)\xi(h^{-1}g)\,d\lambda^{r(g)}(h)
\end{equation}
for all $f\in C_c(G)$, $\xi\in L^2(G,\nu^{-1})$, and $g\in G$
(\cite[Definition II.1.8]{R80}; see also \cite[Proposition 1.41]{W}).

If $f\in C_c(G)$, then its reduced norm is
\begin{equation}
  \label{eq:red_norm}
  \Vert f\Vert_r:=\sup\bigl\{\,\Vert \Ind \delta_u(f)\Vert\,:\,u\in G^{(0)}\,\bigr\},
\end{equation}
where $\delta_u$ is the point mass at  $u\in
G^{(0)}$. The reduced $C^*$-algebra of $G$, $C_r^*(G)$, is the
completion of $C_c(G)$ under the reduced norm.  If $\mu$ is any 
Radon measure on $G^{(0)}$ with full support then $\Vert
f\Vert_r=\Vert (\Ind\mu)(f)\Vert$ for all $f\in C_c(G)$ (\cite[Corollary
5.23]{W}).

\end{dfn}

Recall (\cite[Definition II.1.6]{R80}; see also \cite[Definition
7.7]{W},\cite{R87,R91}) that a unitary representation of a groupoid
$G$ 
with Haar system $\lambda=\{\lambda^u\}_{u\in G^{(0)}}$ is a triple
$(\mu,G^{(0)}*\Hh,\hat{L})$ consisting of a quasi-invariant measure
$\mu$ on $G^{(0)}$, a Borel Hilbert bundle $G^{(0)}*\Hh$ over
$G^{(0)}$, and a Borel homomorphism $\hat{L}:G\to
\text{Iso}(G^{(0)}*\Hh)$ such that $\hat{L}(g)=(r(g),L_g,s(g))$ and
$L_g:\Hh(s(g))\to \Hh(r(g))$ is a Hilbert space isomorphism. Here
$\Hh(u)$ denotes the fiber over $u\in G^{(0)}$. 
Given a Borel Hilbert bundle
$G^{(0)}*\Hh$ and a measure $\mu$ on $G^{(0)}$, we can define the
Hilbert space
\[
  L^2(G^{(0)}*\Hh,\mu)=\{f\in B(G^{(0)}*\Hh)\,:\,u\mapsto \Vert
  f(u)\Vert^2_{\Hh(u)}\text{ is } \mu-  \text{integrable}\},
\]
where $B(G^{(0)}*\Hh)$ is the set of Borel sections of the bundle (see
\cite[Section 3.5]{W} for an outline of Borel bundles and \cite[Appendix
F]{Wcp} for a detailed study of them).

Given a 
unitary representation $(\mu,G^{(0)}*\Hh,\hat{L})$ of $G$
there is an $I$-norm bounded representation $L$ of $C_c(G)$ on
$L^2(G^{(0)}*\Hh,\mu)$ via the vector-valued integral
\[
  L(f)\xi(u)=\int_G f(g)L_g\xi(s(g))\Delta_{\mu}(g)^{-1/2}\,d\lambda^{u}(g)
\]
for $\xi\in L^2(G^{(0)}*\Hh,\mu)$, where $\Delta_{\mu}$ is the modular
function defined by $\mu$. 
The representation $L$ is called
the \emph{integrated form} of the unitary representation (see, for
example, \cite[Definition 7.14]{W}). Moreover,  by the powerful 
disintegration theorem of Renault (\cite{R87}; see also \cite[Theorem 8.2]{W})
 any such representation of $C_c(G)$ is equivalent to the
 integrated form of a unitary representation of $G$.
 
 \begin{rmk}\label{rem:equiv_rep}
 Two unitary representations $L=(\mu, G^{(0)}*\Hh,\hat{L})$ and $L'
 =(\mu, G^{(0)}*\Hh',\hat{L}')$ of $G$ having the same quasi-invariant
 measure $\mu$ are equivalent, $L\cong L'$,
 if $\go*\Hh$ and $\go*\Hh'$ are isomorphic as Hilbert bundles (see,
 for example, \cite[Definition F.22]{Wcp}) via a Borel bundle map
 $U:G^{(0)}*\Hh\to G^{(0)}*\Hh'$ which intertwines $L$ and $L'$ in the
 sense that $U(r(g))\circ L_g=L'_g\circ U(s(g))$ for all $g\in
 G$ (\cite[Definition II.1.6]{R80}). Recall that $U$ is determined by a family of unitaries
 $U(u):\Hh(u)\to \Hh'(u)$ for all $u\in \go$.
 
 Given two unitary representations $L=(\mu, G^{(0)}*\Hh,\hat{L})$ and $L' =(\mu, G^{(0)}*\Hh',\hat{L}')$, we can construct their direct sum  $L'\oplus    L'=(\mu, G^{(0)}*(\Hh\oplus\Hh'), \widehat{L\oplus L'})$ and their tensor product  $L\otimes  L'=(\mu, G^{(0)}*(\Hh\otimes \Hh'), \widehat{L\otimes L'})$ by taking
 \[\widehat{L\oplus L'}(g)=(r(g), L_g\oplus L'_g,s(g)),\;\; \widehat{L\otimes L'}(g)=(r(g), L_g\otimes L'_g,s(g)),\]
 where the direct sums and the tensor products of the Hilbert bundles are done fiberwise.
\end{rmk}
\begin{example}\label{ex:trivialrep}
The
  trivial representation $\iota=(\mu,G^{(0)}\times \CC,i_G)$ on $\mu$,
  where $\mu$ is a quasi-invariant measure, 
  $\go\times \CC$ is the trivial one-dimensional line bundle and
  $(i_G)_g(z)=z$ for all $z\in \CC$. Note that $L\otimes \iota\cong L$
  for all unitary representations $L$ of $G$ with the same
  quasi-invariant measure $\mu$.

\end{example}

 \begin{example}\label{ex:regular_rep}
   Assume that $\mu$ is a quasi-invariant measure on $G^{(0)}$. Let
   $L^2(G,\lambda):=\{L^2(G^u,\lambda^u)\}_{u\in G^{(0)}}$. The (left)
   regular representation $\rho$ of $G$ on $\mu$ is the unitary
   representation $(\mu,G^{(0)}*L^2(G,\lambda),\hat{\rho})$, where
   $$\rho_g:L^2(G^{s(g)},\lambda^{s(g)})\to L^2(G^{r(g)},\lambda^{r(g)})$$ is defined
   via $\rho_g(\xi)(h)=\xi(g^{-1}h)$ for all $\xi\in
   L^2(G^{s(g)},\lambda^{s(g)})$ and $h\in G^{r(g)}$. Even though in general $\rho$ depends on $\mu$, to ease the notation we  write $\rho$ instead of $\rho^\mu$, especially when the measure $\mu$ is fixed.
   
   Its integrated
   form is called the (left) regular representation of $C_c(G)$ on
   $\mu$ and it is unitarily equivalent with $\Ind \mu$ defined in
   \eqref{eq:ind_mu} via $W:L^2(G, \nu)\to L^2(G, \nu^{-1})$,
   $W\xi=\xi\Delta_\mu^{1/2}$ (\cite[Proposition II.1.10]{R80}; see also
   \cite[Definition 3.29 and Exercise 3.30]{Mu_CBMS}). Therefore, if
   $\mu$ has full support,
   $\Vert f\Vert_r=\Vert \rho(f)\Vert$ for all $f\in C_c(G)$.
 \end{example}
 
 Recall that if $A$ is a $C^*$-algebra, $\pi$ is a representation of $A$ and $S$ is a set of representations of $A$, the following assertions are equivalent:

(1) $\ker\pi\supseteq \bigcap\{\ker \sigma\;\mid\; \sigma\in S\}$;

(2) each vector state associated with $\pi$ is a weak-$\ast$ limit of states that are sums of vector functionals associated to representations in $S$.

If either assertion holds, we say that $\pi$ is weakly contained in $S$ and write $\pi\prec S$. If $S=\{\sigma\}$ has only one element, we say that $\pi$ is weakly contained in $\sigma$ and write $\pi\prec \sigma$. In this case there is a surjective homomorphism $C^*(\sigma)\to C^*(\pi)$ given by $\sigma(a)\mapsto \pi(a)$, where $C^*(\pi)$ is the $C^*$-algebra generated by $\pi(a)$ for $a\in A$.
We say that $\pi$ and $\sigma$ are weakly equivalent if and only if $\pi\prec \sigma$ and $\sigma\prec \pi$; this happens if and only if $\ker \pi=\ker \sigma$, and in this case $C^*(\pi)\cong C^*(\sigma)$.

The following definitions and results about amenability  are taken from \cite{AR} and \cite{AR01};  see also chapter 9 in \cite{W}.

\begin{dfn} (see Definition 2.6 in \cite {AR01} and Proposition 2.2.6 in \cite{AR})

Let $G$ be a locally compact groupoid with Haar system $\{\lambda^u\}$. A quasi-invariant measure $\mu$ on $G^{(0)}$ is amenable if there is a net $\{f_i\}$ of non-negative measurable functions on $G$ such that

(1) For all $i$ and a.e. $u\in G^{(0)}$ we have $\ds \int_G f_id\lambda^u=1$;

(2) The functions $\ds g\mapsto \int_G|f_i(g^{-1}h)-f_i(h)|d\lambda^{r(g)}(h)$ tend to the zero function in the weak-$\ast$-topology of $L^\infty(G,\mu\circ\lambda)$.

The groupoid is called measurewise amenable in case each quasi-invariant measure on $G^{(0)}$ is amenable.

\end{dfn}

\begin{dfn}
We say that a locally compact groupoid $G$ is topologically amenable if it admits a continuous approximate invariant mean, i.e. a net $\{m_i^u\}$ of probability measures on $G^u$ for each $u\in G^{(0)}$ which is approximately invariant, in the sense that the function $g\mapsto \|gm_i^{s(g)}-m_i^{r(g)}\|_1$ tends to zero uniformly on the compact subsets of $G$, where $\|\cdot\|_1$ denotes the total variation norm.

\end{dfn}
When $G$ admits a continuous Haar system $\{\lambda^u\}$, we have
\begin{prop}
A locally compact groupoid $G$ with Haar system $\{\lambda^u\}$ is topologically amenable if and only if there exists a net $\{f_i\}$ of non-negative continuous functions on $G$ such that

(i) For all $i$ and $u\in G^{(0)}$ we have $\ds \int_G f_id\lambda^u=1$;

(ii) The functions $\ds g\mapsto \int_G |f_i(g^{-1}h)-f_i(h)|d\lambda^{r(g)}(h)$ tend to the zero function  uniformly on the compact sets of $G$.
\end{prop}

\begin{example}
The Renault-Deaconu groupoid $G(X,T)$ constructed from a local homeomorphism $T:X\to X$ as in Example \ref{RD} is topologically amenable, see \cite{R00} and \cite[Proposition 3.1]{ R15}. 
\end{example}

Note that amenability for groupoids is equivalent to the weak containment of the trivial representation in the regular representation.
\bigskip


\section{Inducing representations from $G\ltimes X$ to $G$ and the
  Koopman representation}
\label{sec:induc-repr-from}
\bigskip

Assume now that the groupoid $G$ acts (on the left) on $X$ and let $G\ltimes X$ be the
action groupoid. We assume as in the previous section that
$\lambda=\{\lambda^u\}_{u\in G^{(0)}}$ is a Haar system on $G$ and the
corresponding Haar system on $G\ltimes X$ is denoted by
$\ol=\{\ol^x\}_{x\in X}$ (see Remark \ref{rm:haarsystem}). Following
the well known  case of the group action groupoid (see,
for example, \cite[Page 5]{AD03} for group actions and \cite[Page 17]{R91} for groupoid
dynamical systems) we define the induction map from unitary
representations of $G\ltimes X$ to unitary representations of $G$.  Let
$(\mu,X*\Hh,\hat{L})$ be a unitary representation of $G\ltimes
X$ and let $\tm =\omega_*(\mu)$. Since $\mu$ is quasi-invariant
for $G\ltimes X$, Theorem \ref{thm:quasi-invariant} implies that
$\tm $ is quasi-invariant for $G$ and there is a decomposition
of $\mu$ relative to $\omega$ such that $g\mu_{s(g)}\sim
\mu_{r(g)}$. Let $D:G\ltimes X\to \RR^+$ be a Borel choice of the
Radon-Nikodym derivative $\ds \frac{d g\mu_{s(g)}}{d\mu_{r(g)}}$ such that
$D(g_1g_2,x)=D(g_1,g_2\cdot x)D(g_2, x)$ for all 
   $(g_1,g_2)\in G^{(2)}$ and $\mu$-almost all $x$ (see the proof of Theorem \ref{thm:quasi-invariant} ). 
For each $u\in G^{(0)}$ define \[ \Kk(u):=L^2(X_u,\mu_u):=\int_{X_u}^\oplus
\Hh(x)\,d\mu_u(x),\]
where $X_u=\omega^{-1}(u)$. The Borel structure on $\Hh$ defines  a natural
Borel structure on $\Kk:=\{\Kk(u)\}_{u\in G^{(0)}}$
making $G^{(0)}*\Kk$ a Borel Hilbert bundle.

\begin{dfn}\label{dfn:inducedrep}
  The \emph{induced representation} of  a unitary representation
  $(\mu,X*\Hh,\hat{L})$ of $G\ltimes X$ to $G$
  is the unitary representation
  $(\tm ,G^{(0)}*\Kk,\Ind\hat{L})$ of $G$,
  where $\Ind\hat{L}:G\to \operatorname{Iso}(G^{(0)}*\Kk)$,
  $\Ind\hat{L}_{g}=(r(g),\Ind L_{g},s(g))$  and, for $g\in G$,
  $\Ind L_{g}:\Kk({s(g)})\to \Kk({r(g)})$ is
  defined via
  \[\Ind L_{g}\xi(x)=D(g^{-1},x)^{1/2}L_{(g,g^{-1}\cdot x)}\bigl(
    \xi(g^{-1}\cdot x)\bigr)\]
  for all $\xi \in \Kk(s(g))$.
\end{dfn}

\begin{rmk}
  Note that the above Definition can be deduced from \cite[Sections
  3.2 and 3.3]{R14} with a bit of effort. Indeed, let $Z=G\ltimes X$
  viewed as a topological space. Then $G$ acts properly on the left on
  $Z$ via the natural action, and $G\ltimes X$ acts properly and freely
  on the right on $Z$, since the action of any groupoid on itself is
  free and proper.  Moreover, one can check that $Z$ is a groupoid
  correspondence in the sense of Holkar (see, for example, Definition
  2.3 of \cite{R14}), with the cocycle $\Delta$ in the definition
  being the Radon-Nykodim derivative on $G\ltimes X$ corresponding to
  $\mu$ and the system of measures $\alpha=\{\alpha_u\}$ given by
  $\alpha_u=\lambda_u$
  for all $u\in \go$. Then, following the steps in Section 3.2 and 3.3
  of \cite{R14}, one can recover our Definition \ref{dfn:inducedrep}.
\end{rmk}

As discussed above, the induced representation of $G$ extends to an $I$-bounded
representation  $\Ind L:C_c(G)\to
\Bb(L^2(G^{(0)}*\Kk,\tm ))$  via the vector integral
\begin{align*}
  \Ind L(f)\xi(u)&=\int_{G^{u}}
                                f(g)\Ind L_g\xi(s(g))\Delta_{\tm}(g)^{-1/2}\,d\lambda^u(g)
\end{align*}
for all $\xi \in L^2(G^{(0)}*\Kk,\tm)$.
Equivalently, the induced representation is characterized by (see
\cite[Proposition 7.12]{W})
\begin{multline*}
  \la \Ind L(f)\xi\,,\,\eta \ra\\
  =\int_{G^{(0)}}\int_{G^u}f(g) \bigl\langle 
  \Ind L_g\xi(s(g))\,,\,
  \eta(r(g))\bigr\rangle\Delta_{\tm}(g)^{-1/2}\,d\lambda^u(g)\,d\tm (u)\\
  =\int_{G^{(0)}}\int_{G^u}f(g)\int_{X_u}D(g^{-1},x)^{1/2}\bigl\langle
  L_{(g,g^{-1}\cdot x)}\xi(s(g))(g^{-1}\cdot
  x)\,,\,\eta(r(g))(x)\bigr\rangle \\
   \,d\mu_u(x) \Delta_{\tm}(g)^{-1/2}\,d\lambda^u(g)\,d\tm (u)\\
  =\int_{G^{(0)}}\int_{X_u}\int_{G^u}f(g)\Delta_\mu{(g^{-1},x)}^{1/2}\bigl\langle
  L_{(g,g^{-1}\cdot x)}\xi(s(g))(g^{-1}\cdot
  x)\,,\,\eta(r(g))(x)\bigr\rangle\\
 \,d\lambda^{u}(g)d\mu_u(x)d\tm (u)\\
  =\int_X\int_{G^{\omega{(x)}}}f(g)\Delta_\mu{(g^{-1},x)}^{1/2}\bigl\langle
  L_{(g,g^{-1}\cdot x)}\xi(s(g))(g^{-1}\cdot
  x)\,,\,\eta(r(g))(x)\bigr\rangle\\
d\lambda^{\omega(x)}(g)\,d\mu(x),\\
\end{multline*}
for all $\xi,\eta \in L^2(G^{(0)}*\Kk,\tm)$, where $\bigl\langle \cdot\,,\,\cdot \bigr\rangle$ represent the inner-products in the
corresponding Hilbert spaces and  $\Delta_{\mu}(g^{-1},x)=D(g^{-1},x)\Delta_{\tm}(g^{-1})$ is the
Radon-Nikodym derivative on $G\ltimes X$ corresponding to $\mu$.

\begin{dfn}
Assuming that the groupoid $G$ acts on $X$, let $\mu$ be a $G$-quasi-invariant measure on $X$ or, equivalently,
  a quasi-invariant measure for $G\ltimes X$. We define the Koopman
  representation $\kappa^\mu$ of $G$ to be the induced 
  representation  of the trivial representation $(\mu,X\times
  \CC,i_{G\ltimes X})$,  where, recall from Example
  \ref{ex:trivialrep},  $(i_{G\ltimes X})_{(g,x)}(z)=z$ for all
  $(g,x)\in G\ltimes   X$ and $z\in \CC$. Since the measure $\mu$ is
  typically fixed, we write shortly $\kappa$ for $\kappa^\mu$ when there is no
  possibility for confusion.  In general, the relationship between $\kappa$ and $\mu$ is complicated, and we plan to address this issue in a future project. In the last section, we illustrate this relationship in some particular examples.
\end{dfn}

Therefore the Koopman representation $\kappa$ is given by
$(\tm ,G^{(0)}*\Kk,\hat{\kappa})$, where 
$\Kk=\{L^2(X_u,\mu_u)\}_{u\in G^{(0)}}$ and, for $g\in G$,
\[\kappa_g:L^2(X_{s(g)},\mu_{s(g)})\to
L^2(X_{r(g)},\mu_{r(g)})\] is given by
\[
  \kappa_g\xi(x)=D(g^{-1},x)^{1/2}\xi(g^{-1}\cdot x),
\]
which recovers the classical definition for group actions (see Definition 13.A.5 in \cite{BH} for example).

Hence $\kappa$ extends to an $I$-bounded representation of
$C_c(G)$ on $L^2(G^{(0)}*\Kk,\tm )$ via
\[
  \kappa(f)\xi(u)=\int_{G^{u}}f(g)\kappa_g\xi(s(g))\Delta_{\tm}(g)^{-1/2}\,d\lambda^u(g).
\]
Note that we can identify $L^2(G^{(0)}*\Kk,\tm )$ with
$L^2(X,\mu)$ via the unitary $V:L^2(G^{(0)}*\Kk,\tm )\to
L^2(X,\mu)$, $V(\xi)(x)=\xi(\omega(x))(x)$. 
Therefore we can view the Koopman  representation as a representation
of $C_c(G)$ on $L^2(X,\mu)$ via
\begin{align*}
  \kappa(f)\xi(x)&=\int_{G^{\omega(x)}}f(g)\kappa_g(\xi)(x)\Delta_{\tm}(g)^{-1/2}\,d\lambda^{\omega(x)}(g)\\
  &=\int_{G^{\omega{(x)}}}f(g)\xi(g^{-1}\cdot
    x)\Delta_{\mu}(g^{-1},x)^{1/2}\,d\lambda^{\omega{(x)}}(g), 
\end{align*}
where recall that $\Delta_{\mu}(g^{-1},x)=D(g^{-1},x)\Delta_{\tm}(g^{-1})$.
Equivalently, $\kappa$ is characterized by
\begin{multline*}
  \langle  \kappa(f)\xi\; ,\;\eta \rangle=\int_X
  \int_{G^{\omega(x)}}f(g)\xi(g^{-1}\cdot
                                      x)\overline{\eta(x)}\Delta_{\mu}(g^{-1},x)^{1/2}d\lambda^{\omega(x)}(g)d\mu(x)\\
  =\int_{G^{(0)}}\int_{X_u} \int_{G^u} f(g)\xi(g^{-1}\cdot
                                      x)\overline{\eta(x)}\Delta_{\mu}(g^{-1},x)^{1/2}d\lambda^{u}(g)d\mu_u(x)d\tm(u)
\end{multline*}
for all $\xi,\eta\in L^2(X,\mu)$.

We denote by $C^*(\kappa)$  the closure of $\kappa(C_c(G))$ in the operator norm of $\Bb(L^2(X,\mu))$.

\begin{example}
Let $(G,E)$ be a level transitive self-similar groupoid action (see \cite{D22}) such that $|uE^1|=p\ge2$ is constant for all $u\in E^0$. Then $G$ acts on $X=E^\infty$ and the uniform probability measure $\nu$ on $E^\infty$ is $G$-invariant. Then the $C^*$-algebra $C^*(\kappa)$ of the Koopman representation of $G$ on $L^2(X,\nu)$ is residually finite dimensional and it has a normalized trace $\tau_0$. For $G$ an amenable group and for an essentially free self-similar action, it is proved in Theorem 9.14 of \cite{G} that  $C^*(\kappa)\cong C^*_r(G)$. We believe that this isomorphism holds true for level transitive self-similar amenable groupoid actions.
\end{example}

\bigskip

\begin{rmk}
If $X=G^{(0)}$ and   $G$ acts on $X$
  via $g\cdot s(g)=r(g)$, then, since $G\ltimes G^{(0)}\cong G$,  the Koopman representation $\kappa^\mu$ of $G$ associated to a $G$-quasi-invariant measure  $\mu$ on
  $G^{(0)}$ is given by the
  trivial representation $(\mu,G^{(0)}\times \CC,i_G)$.

When $X=G$ and $G$ acts on itself by left multiplication, the Koopman representation  $\kappa^\mu$ is just the left regular representation $\rho=\rho^G$. If $X=G/H$ where $H$ is a closed subgroupoid, the Koopman representation is the quasi-regular representation $\rho^{G/H}$.
\end{rmk}

\begin{rmk}
Given a closed subgroupoid $H$ of $G$ with the same unit space, recall
that $G$ acts on $G/H$ with $\omega:G/H\to G^{(0)}, \omega(gH)=r(g)$.
Given a unitary representation $L=(\nu, H^{(0)}*\Kk,\hat{L})$ of $(H,\beta)$,
one can induce it to a representation of $(G,\lambda)$ following the steps in
\cite[Section 3]{R14}. Specifically, under the assumption that
$H^{(0)}=G^{(0)}$, one can define a groupoid correspondence from
$(G,\lambda)$ to $(H,\beta)$ in the sense of \cite[Definition
2.3]{R14} by setting $X=G$ and $\alpha_u=\lambda_u$ for all $u\in
\go$. Therefore,  $\nu$ defines  a $G$-quasi-invariant measure
$\mu$  on $G/H$ as in \cite[Section 3.2]{R14}. Let $\tm$ and
$\{\mu_u\}_{u\in\go}$ as in Definition 
\ref{def:qi}. Then the induced representation $\Ind_H^G L$ of $G$ is
$(\tilde\mu, G^{(0)}*\Hh,\Ind_H^G \hat{L})$,  where $\Hh$ is the
Hilbert bundle obtained from the completion of \[\{\xi:G\to \Kk:
  \xi(g)\in \Kk(s(g))\;\text{and}\;\xi(gh)=L_{h^{-1}}\xi(g)\}\] and  
\[(\Ind_H^G L)_g\xi(x)=D(g^{-1},xH)^{1/2}\xi(g^{-1}x).\]
We have $\ds \Hh(u)=L^2((G/H)_u*\Kk, \mu_u)=\int_{(G/H)_u}^{\oplus}\Kk(x)d\mu_u(x)$, where $(G/H)_u=\{gH\in G/H: r(g)=u\}$.
The induced representation of a direct sum is the direct sum of induced representations. 
\end{rmk}

\begin{example}
If $H=G^{(0)}$ and $\iota=(\mu, G^{(0)}\times \mathbb C, i_H)$ is the trivial representation of $H$ with $(i_H)_u(z)=z$, then $\Ind_H^G \iota$ is the left regular representation $\rho^G$ of $G$. For a general closed subgroupoid $H$, $\Ind_H^G\iota$ is the quasi-regular representation $\rho^{G/H}$ of $G$ on $L^2(G/H, \mu)$.
\end{example}

The following result is inspired from the similar result in the case
of groups, see \cite[Appendix E]{BHV}. 
\begin{prop}\label{E.2.5}
Suppose $H$ is a   closed subgroupoid of $G$ with the same unit space.     Let $L=(\mu, G^{(0)}*\Hh,\hat{L})$ be a unitary representation of $G$ and let $M=(\mu,H^{(0)}*\Kk,\hat{M})$ be a unitary representation of $H$. Then $L\otimes\Ind_H^G M$ is equivalent to $\Ind_H^G((L|_H)\otimes M).$
\end{prop}

\begin{proof}
If $G^{(0)}*\Mm$ and $G^{(0)}*\Ll$ are the Hilbert bundles of
$\Ind_H^G M$ and $\Ind_H^G((L|_H)\otimes M)$ respectively, define a
Borel bundle map $U:G^{(0)}*(\Hh\otimes \Mm)\to G^{(0)}*\Ll$ by
\[U(u)(\xi\otimes \eta)(x)=L_{x^{-1}}\xi\otimes\eta(x),\; \forall
  u\in\go\text{ and } x\in G_u\]
and verify that $U(u)$ is unitary for each $u$. Moreover, $U$ intertwines $\Ind((L|_H)\otimes M)$ and $L\otimes\Ind_H^G M$ since
\[\left(\left(\Ind_H^G(L|_H\otimes M)_g\right)U(s(g))(\xi\otimes\eta)\right)(x)=D(g^{-1},xH)^{1/2}L_{x^{-1}g}\xi\otimes\eta(g^{-1}x)=\]\[=U(r(g))\left(L_g\xi\otimes(\Ind_H^GM)_g\eta\right)(x).\]
\end{proof}

\begin{cor}\label{E.2.6}

Let $G$ be a locally compact groupoid and let $H$ be a closed subgroupoid with the same unit space. If $L=(\mu, G^{(0)}\ast \Hh, \hat{L})$ is a representation of $G$, then $\Ind_H^G(L|_H)$ is equivalent to $L\otimes \rho^{G/H}$, where $\rho^{G/H}$ is the quasi-regular representation of $G$ on $L^2(G/H, \mu)$. In particular, for $H=G^{(0)}$, $L\otimes \rho^G$ is equivalent to $(\dim L)\otimes \rho^G$, where $(\dim L)_u=\text{id}_{\Hh(u)}$.

\end{cor}

\begin{proof}
For the first part, we apply Proposition \ref{E.2.5} for $M$ the trivial representation of $H$. For the second part, $L|_{G^{(0)}}$ is a direct sum of trivial representations on  $G^{(0)}*\Cc$, where for $u\in G^{(0)}$, $\Cc(u)=\mathbb C^{n(u)}$ if $n(u)=\dim\Hh(u)$ is finite and $\Cc(u)$ is infinite dimensional otherwise.
\end{proof}

\bigskip

\section{Properties of the Koopman representation}

\bigskip

We still assume that the groupoid $G$ acts (on the left) on $X$ and let $G\ltimes X$ denote the
action groupoid. 

\begin{lem}\label{lem:normind}
Let $(\mu,X*\Hh,\hat{L})$ be a unitary
  representation of $G\ltimes X$. Then for all non-negative $f\in C_c(G)$ we have
  \[
    \Vert \Ind L(f)\Vert \le \Vert \kappa^\mu(f)\Vert.
  \]
\end{lem}

\begin{proof}
   For $\xi\in L^2(G^{(0)}*\Kk)$ define $\tilde{\xi}(x)=\Vert \xi( \omega(x))(x)\Vert$.
  Then $\tilde{\xi}\in L^2(X,\mu)$ and
  $\Vert \tilde{\xi}\Vert=\Vert \xi\Vert$.
  
  Let $f\in C_c(G)$ be a non-negative function, and let $\xi,\eta\in
  L^2(G^{(0)}*\Kk)$. We have
  \begin{multline*}
    \left\vert \langle  \Ind L(f)\xi\;,\;\eta
      \rangle\right\vert=\\
    \left\vert
      \int_{G^{(0)}}\int_{G^u}\int_{X_u}f(g)\Delta_{\mu}(g^{-1},x)^{1/2}\bigl\langle
      L_{(g,g^{-1}\cdot x)}\xi(s(g))(g^{-1}\cdot x)\,,\,\eta(r(g))(x)
      \bigr\rangle\right.\\
   \,\,\left. d\mu_u(x)d\lambda^u(g)d\tm(u)\right\vert\\
    \le \int_{G^{(0)}}\int_{G^u}\int_{X_u}f(g)\Delta_{\mu}(g^{-1},x)^{1/2}\left\vert\bigl\langle
      L_{(g,g^{-1}\cdot x)}\xi(s(g))(g^{-1}\cdot x)\,,\,\eta(r(g))(x)
      \bigr\rangle\right\vert\\
    \,\, d\mu_u(x)d\lambda^u(g)d\tm(u)\\
  \end{multline*}
  which, since $L_{(g,g^{-1}\cdot x)}$ is a Hilbert
  space isomorphism,
  \begin{multline*}
    \le \int_{G^{(0)}}\int_{G^u}\int_{X_u}f(g)\Delta_{\mu}(g^{-1},x)^{1/2}
                        \Vert\xi(s(g))(g^{-1}\cdot x)\Vert\Vert\eta(r(g))(x)\Vert
                       \\
                      d\mu_u(x)d\lambda^u(g)d\tm(u)\\
    = \int_{G^{(0)}}\int_{G^u}\int_{X_u}f(g)\Delta_{\mu}(g^{-1},x)^{1/2}
                        \tilde{\xi}(g^{-1}\cdot x)\tilde{\eta}(x)                       
      d\mu_u(x)d\lambda^u(g)d\tm(u)\\
    =\langle  \kappa_\mu(f)\tilde{\xi}\;,\;\tilde{\eta} \rangle\le \Vert
      \kappa_\mu(f)\Vert\Vert \tilde{\xi}\Vert\Vert\tilde{\eta}\Vert
    =\Vert
      \kappa_\mu(f)\Vert\Vert \xi\Vert\Vert\eta\Vert
  \end{multline*}
 The result follows.
\end{proof}

\medskip

\begin{thm}\label{thm:weak_Koopman}
  With the notation as above, assume that $\mu$ is a
  $G$-quasi-invariant probability measure on $X$ with full support and let $\tm$
  be the push-forward quasi-invariant measure on $G^{(0)}$. Then, for all non-negative $f\in
  C_c(G)$ we have $\Vert \Ind \tm (f)\Vert \le \Vert
  \kappa^\mu(f)\Vert$. 
\end{thm}

\begin{proof}
  Recall from Example \ref{ex:regular_rep} that $\Ind \tm$ is
  unitarily equivalent with the integrated form of the unitary
  representation $\rho=(\tm,G^{(0)}*L^2(\lambda),\hat{\rho})$ of $G$, where
  $L^2(\lambda)=\{L^2(G^u,\lambda^u)\}_{u\in G^{(0)}}$, and, for $g\in
  G$, $\rho_g(\xi)(h)=\xi(g^{-1}h)$ for all $h\in G^{r(g)}$.

  Consider the unitary representation $L=(\mu,X*\Hh,\hat{L})$ of $G\ltimes X$, where
  $\Hh(x)=L^2(G^{\omega(x)},\lambda^{\omega(x)})$ for all $x\in X$,
  and, for $(g,x)\in G\ltimes X$, $L_{(g,x)}:\Hh(x)\to \Hh(g\cdot x)$ is given by
  \[
    L_{(g,x)}(\xi)(h)=\xi(g^{-1}h)\,\text{ for all }\xi\in
    \Hh(x)\text{ and }h\in G^{r(g)}.
  \]
  Let $f\in C_c(G)$ be a non-negative function. By Lemma
  \ref{lem:normind}, $\Vert \Ind L(f)\Vert \le \Vert
  \kappa_\mu(f)\Vert$. We prove next that $\Vert \rho(f)\Vert \le \Vert \Ind
  L(f)\Vert$. This implies the result.

  By definition, $\Ind L$ is given by $(\tm,G^{(0)}*\Kk,\Ind
  \hat{L})$, where \[\Kk(u)=\int_{X_u}^\oplus
  L^2(G^u,\lambda^u)\,d\mu_u(x).\] 
  While $\Kk(u)\simeq
  L^2(X_u,\mu_u)\otimes L^2(G^u,\lambda^u)$, we prefer to
  view elements of $\Kk$ as sections $\xi:X_u\to
  L^2(G^u,\lambda^u)$ endowed with the norm
  \[
    \Vert \xi\Vert^2=\int_{X_u}\int_{G^u}\vert \xi(x)(g)\vert^2\,d\lambda^u(g)d\mu_u(x).
  \]
  Then, for $g\in G$, $\xi \in \Kk(s(g))$, $x\in X_{r(g)}$
  and $h\in G^{r(g)}$,
  \begin{align*}
    \Ind L_g(\xi)(x)(h)&=D(g^{-1},x)^{1/2}L_{(g,g^{-1}\cdot
                         x)}(\xi(g^{-1}\cdot x))(h)\\
    &=D(g^{-1},x)^{1/2}\xi(g^{-1}\cdot x)(g^{-1}h),
  \end{align*}
  for all $\xi\in \Kk(s(g))$.
  Therefore, for $f\in C_c(G)$, $\Ind L(f)$ acts on $L^2(G^{(0)}*\Kk,\tm)$ via
  \begin{multline*}
    \la \Ind L(f)\xi\,,\,\eta
    \ra=\int_{G^{(0)}}\int_{G^u}f(g)\bigl\langle \Ind
            L_g\xi(s(g))\,,\,\eta(r(g))
            \bigr\rangle\Delta_{\tm}(g)^{-1/2}\,d\lambda^u(g)d\tm(u)\\
    =\int_{G^{(0)}}\int_{G^u}\int_{X_u}\int_{G^u}f(g)\xi(s(g))(g^{-1}\cdot
      x)(g^{-1}h)\overline{\eta(u)(x)(h)}\Delta_{\mu}(g^{-1},x)^{1/2}\\
 \,d\lambda^u(h)\,d\mu_u(x)\,d\lambda^u(g)d\tm(u),\\
  \end{multline*}
  for all $\xi,\eta\in L^2(G^{(0)}*\Kk,\tm)$.

  Let $f\in C_c(G)$ be a non-negative function and let $\xi,\eta\in
  L^2(G^{(0)}*L^2(\lambda),\tm)$. Then $\xi$ defines an element
  $\tilde{\xi}\in L^2(G^{(0)}*\Kk,\tm)$ via
  $\tilde{\xi}(u)(x)(h):=\xi(u)(h)$ and $\Vert \xi\Vert=\Vert
  \tilde{\xi}\Vert$ since $\mu_u$ is a probability 
  measure for all $u\in G^{(0)}$. 
  Similarly $\eta$ defines
  $\tilde{\eta}\in L^2(G^{(0)}*\Kk,\tm)$ such that $\Vert
  \eta\Vert=\Vert \tilde{\eta}\Vert$.  We have
  \begin{multline*}
    \left\vert \la \rho(f)\xi\,,\,\eta \ra\right\vert=\left\vert
                                                          \int_{G^{(0)}}\int_{G^u}f(g)\bigl\langle
                                                         \rho_g(\xi(s(g)))\,,\,\eta(u)
                                                          \bigr\rangle\Delta_{\tm}(g)^{-1/2}\,d\lambda^u(g)\,d\tm(u)\right\vert\\
    =\left\vert
      \int_{G^{(0)}}\int_{G^u}\int_{G^u}f(g)\xi(s(g))(g^{-1}h)\overline{\eta(u)(h)}\Delta_{\tm}(g)^{-1/2}
      d\lambda^u(h)d\lambda^u(g)d\tm(u)\right\vert
  \end{multline*}
  which, since
    $\ds \Delta_{\tm}(g^{-1})=\int_{X_u}\Delta_{\mu}(g^{-1},x)d\mu_{u}(x)$
    a.e. and $\Delta_{\tm}$ is a cocycle, 
    \begin{multline*}
   = \left\vert
      \int_{G^{0}}\int_{G^u}\int_{X_u}\int_{G^u}f(g)\tilde{\xi}(s(g))(g^{-1}\cdot
      x)(g^{-1}h)\overline{\tilde{\eta}(u)(x)(h)}\Delta_{\mu}(g^{-1},x)^{1/2}\right.\\
                                                       \left.  d\lambda^{u}(h)d\mu_u(x)d\lambda^u(g)d\tm(u)\right\vert\\
    =\left\vert \la \Ind L(f)\tilde{\xi}\,,\,\tilde{\eta}
      \ra\right\vert\le \Vert \Ind L(f)\Vert \Vert
      \tilde{\xi}\Vert\Vert\tilde{\eta}\Vert
    =\Vert \Ind L(f)\Vert \Vert\xi\Vert\Vert \eta\Vert.
  \end{multline*}
  It follows that $\Vert \rho(f)\Vert\le \Vert \Ind L(f)\Vert \le \Vert
  \kappa^\mu(f)\Vert$. 
\end{proof}

\begin{thm}\label{thm:kweakrho}
Assume that the action groupoid $(G\ltimes X,\overline{\lambda})$ is
$\sigma$-compact and amenable.  Assume also that  $\mu$ has full support. Then the Koopman
representation $\kappa^\mu$ is weakly contained in the left regular
representation $\rho$.  In particular,
 we have a surjection $C_r^*(G)\to C^*(\kappa^\mu)$.

\end{thm}
We will write in the following $\rho^G$ for the left regular
representation on $\tm$ of $G$ and $\rho^{G\ltimes X}$ for the left
regular representation on $\mu$ of $G\ltimes X$. We break the proof of
the theorem into two parts. First we prove that 
the Koopman representation $\kappa$ is weakly contained into
the induced representation $\Ind \rho^{G\ltimes X}$. In the second part we prove that
$\Vert \Ind \rho^{G\ltimes X}(f)\Vert\le \Vert \rho^G(f)\Vert$ for all
$f\in C_c(G)$. This implies the result.

Using \cite[Proposition 2.2.7]{AR} (see the discussion following
Definition 2.6 of \cite{R15} for the equivalence between the various
definitions of amenability in the $\sigma$-compact case), there is a
sequence of functions $\{f_n\}\in C_c(G\ltimes X)$ such that the
following conditions hold:
\begin{equation}
  \label{eq:amen1}
  \int_{G^{\omega(x)}}\vert f_n(g,g^{-1}x)\vert^2
  \,d\lambda^{\omega(x)}(g)=1\,\,\text{ for all }\,x\in X
\end{equation}
and
\begin{equation}
  \label{eq:amen2}
  \lim_{n\to\infty}\int_{G^{r(h)}}\vert
 f_n(h^{-1}g,g^{-1}h
  x)-f_n(g,g^{-1}h x)\vert^2\,d\lambda^{r(h)}(g)=0
\end{equation}
uniformly on compact subsets of $G\ltimes X$.

\begin{prop}
  Assume the hypotheses of Theorem \ref{thm:kweakrho}. Then the
  Koopman representation is weakly contained in $\Ind \rho^{G\ltimes
    X}$. Therefore, $\Vert \kappa^\mu(f)\Vert\le \Vert \Ind \rho^{G\ltimes
  X}(f)\Vert$ for all $f\in C_c(G)$.
\end{prop}
\begin{proof}
  The left regular representation $\rho^{G\ltimes X}$ on $\mu$ of
  $G\ltimes X$ is the unitary representation $(\mu,X*L^2(G\ltimes
  X,\overline{\lambda}),\hat{\rho}^{G\ltimes X})$, where the fiber
  over $x\in X$ of the Hilbert bundle is $L^2((G\ltimes
  X)^x,\overline{\lambda}^x)$. We will write $\langle\cdot \,,\, \cdot\rangle_x$
  for the inner product in the fiber over $x$. It is useful to keep in mind that
  \[
    (G\ltimes X)^x=\{(g,g^{-1}x)\,:\,g\in G^{\omega(x)}\}.
  \]
  If $(h,x)\in G\ltimes X$ then $\rho^{G\ltimes
    X}_{(h,x)}:L^2((G\ltimes X)^x,\overline{\lambda}^x)\to
  L^2((G\ltimes X)^{h x},\overline{\lambda}^{h x})$ is given
  by
  \[
    \rho^{G\ltimes X}_{(h,x)}\xi(g,g^{-1}h
    x)=\xi(h^{-1}g,g^{-1}h x).
  \]
  Therefore if $\xi\in L^2((G\ltimes X)^x,\overline{\lambda}^x)$ and
  $\eta\in L^2((G\ltimes X)^{h x},\overline{\lambda}^{h
    x})$,
  \[
    \langle \rho^{G\ltimes X}_{(h,x)}\xi,\eta
    \rangle_x=\int_{G^{r(h)}}\xi(h^{-1}g,g^{-1}h
    x)\overline{\eta(g,g^{-1}h x)}\,d\lambda^{r(h)}(g).
  \]

  Then the induced representation $\Ind \rho^{G\ltimes X}$ is the
  unitary representation $(\tm,\go*\Ll,\Ind \hat{\rho}^{G\ltimes X})$
  of $G$, where \[\Ll(u)=\int_{X_u}^\oplus L^2((G\ltimes
  X)^x,\overline{\lambda}^x)\,d\mu_u(x)\cong L^2(X_u,\mu_u)\otimes L^2(G^u,\lambda^u)\cong\Kk(u).\] 
  Thus, if $\xi,\eta\in
  \Ll(u)$
  \begin{align*}
  \langle \xi\,,\,\eta \rangle_u&=\int_{X_u}\langle
                                 \xi(x),\eta(x) \rangle_x\,d\mu_u(x)\\
    &=\int_{X_u}\int_{G^u}\xi(x)(g,g^{-1}x)\overline{\eta(x)(g,g^{-1}x)}\,d\lambda^u(g)\,d\mu_u(x).
  \end{align*}
  If $h\in G$, $\Ind \rho^{G\ltimes X}_{h}:\Ll(s(h))\to
  \Ll(r(h))$ is given via
  \[
    (\Ind \rho^{G\ltimes X}_{h}\xi)(x)(g,g^{-1}x)=D(h^{-1},x)^{1/2}\xi(h^{-1}x)(h^{-1}g,g^{-1}x)
  \]
  for all $\xi\in\Ll(s(h))$, $x\in X_{r(h)}$ and
  $g\in G^{r(h)}$. Therefore, if $\xi\in \Ll(s(h))$ and
  $\eta\in\Ll(r(h))$ we have
  \begin{multline*}
    \langle \Ind\rho^{G\ltimes X}_{h}\xi\,,\,\eta
    \rangle_{r(h)}\\
    =\int_{X_{r(h)}}\int_{G^{r(h)}}D(h^{-1},x)^{1/2}\xi(h^{-1}x)(h^{-1}g,g^{-1}x)\overline{\eta(x)(g,g^{-1}x)}\,d\lambda^{r(h)}(g)d\mu_{r(h)}(x).
  \end{multline*}
  The integrated form of $\Ind\rho^{G\ltimes X}$ of $C_c(G)$ acts on
  $L^2(\go*\Ll,\tm)$ via
  \begin{multline*}
    \langle \Ind \rho^{G\ltimes X}(f)\xi\,,\,\eta \rangle\\
    =\int_{\go}\int_{G^u}f(h)\int_{X_u}\int_{G^u}D(h^{-1},x)^{1/2}\xi(s(h))(h^{-1}x)(h^{-1}g,g^{-1}x)\\
    \cdot
    \overline{\eta(u)(x)(g,g^{-1}x)}\,d\lambda^u(g)d\mu_u(x)\Delta_{\tm}(g)^{-1/2}\,d\lambda^u(h)d\tm(u) 
  \end{multline*}
  for all $f\in C_c(G)$ and $\xi,\eta\in L^2(\go*\Ll,\tm)$.

  Let $\kappa^\mu$ be the Koopman representation acting on $L^2(X,\mu)$
  and let $\xi\in L^2(X,\mu)$. For $n\in \NN$ define $\xi_n\in
  L^2(\go*\Ll,\tm)$ via
  \[
    \xi_n(u)(x)(g,g^{-1}x)=\xi(x)f_n(g,g^{-1}x).
  \]
  We check that indeed $\xi_n\in L^2(\go*\Ll,\tm)$ for all $n\in\NN$
  and $\Vert \xi_n\Vert=\Vert \xi\Vert$:
  \begin{multline*}
    \Vert \xi_n\Vert^2=\int_{\go}\Vert \xi(u)\Vert_{u}^2\,d\tm(u)\\
    =\int_{\go}\int_{X_u}\int_{G^u}\vert
    \xi_n(u)(x)(g,g^{-1}x)\vert^2\,d\lambda^u(g)d\mu_u(x)d\tm(u)\\
    =\int_{\go}\int_{X_u}\vert \xi(x)\vert^2\left(
      \int_{G^u}\vert
     f_n(g,g^{-1}x)\vert^2\,d\lambda^u(g)
    \right)d\mu_u(x)d\tm(x)\\
    =\int_{\go}\int_{X_u}\vert
    \xi(x)\vert^2\,d\mu_u(x)d\tm(u)=\Vert \xi\Vert^2.
  \end{multline*}
  We used \eqref{eq:amen1} in the second to last equality.

  Next we prove that $\ds \lim_{n\to\infty}\langle  \Ind \rho^{G\ltimes
    X}(f)\xi_n\,,\,\xi_n \rangle=\langle  \kappa^\mu(f)\xi,\xi \rangle$ for
  all $f\in C_c(G)$. This implies the weak containment of $\kappa^\mu$ in
  $\Ind \rho^{G\ltimes X}$.  We have
  \begin{multline*}
    \langle \Ind \rho^{G\ltimes X}(f)\xi_n\,,\,\xi_n
    \rangle=\int_{\go}\int_{G^u}f(h)\int_{X_u}\xi(h^{-1}
    x)\overline{\xi(x)}D(h^{-1},x)\\
    \cdot \left(
      \int_{G^u}f_n(h^{-1}g,g^{-1}x)\overline{f_n(g,g^{-1}x)}
    \,d\lambda^u(g)\right)d\mu_u(x)\Delta_{\tm}(h)^{-1/2}d\lambda^u(h)d\tm(u)
\end{multline*}
Equation \ref{eq:amen2} implies that (see the proof of
\cite[Proposition 2.2.7]{AR})
\begin{multline*}
 2 \lim_{n\to
   \infty}\int_{G^u}f_n(h^{-1}g,g^{-1}x)\overline{f_n(g,g^{-1}x)}d\lambda^u(g)=\\
 \lim_{n\to
 \infty}\left( \int_{G^u}\vert
 f_n(h^{-1}g,g^{-1}x)\vert^2d\lambda^u(h)+\int_{G^u}\vert
f_n(g,g^{-1}x)\vert^2d\lambda^u(g) \right)=2
\end{multline*}
uniformly on compact subsets of $G\ltimes X$. Therefore
  \begin{multline*}
    \lim_{n\to \infty}\langle \Ind \rho^{G\ltimes X}(f)\xi_n\,,\,\xi_n
    \rangle\\
    =\int_{\go}\int_{G^u}f(h)\int_{X_u}\xi(h^{-1}
    x)\overline{\xi(x)}D(h^{-1},x)d\mu_u(x)\Delta_{\tm}(h)^{-1/2}d\lambda^u(h)d\tm(u)
    \\
    =\langle \kappa^\mu(f)\xi\,,\,\xi \rangle.
  \end{multline*}
\end{proof}

\begin{prop}
  Under the hypotheses of Theorem \ref{thm:kweakrho}, $\Vert \Ind
  \rho^{G\ltimes X}(f)\Vert \le \Vert \rho^G(f)\Vert$ for all $f\in C_c(G)$.
\end{prop}
\begin{proof}
 Note that the representation $\Ind \rho^{G\ltimes
   X}$ is equivalent  with  $\Ind L$ of Theorem
 \ref{thm:weak_Koopman}. Indeed, if $u\in \go$,  $V(u):\Ll(u)\to
 \Kk(u)$ defined via
 \[V(u)(\xi)(x)(g)=\xi(x)(g,g^{-1}x)\text{ where }r(g)=\omega(x)=u\] is a
 unitary. Moreover $V$  intertwines $\Ind \rho^{G\ltimes X}$
 and $\Ind L$ in the sense of Remark \ref{rem:equiv_rep}. 

  Therefore $\Ind\rho^{G\ltimes X}\cong \kappa^\mu\otimes \rho^G$ and by
  Corollary \ref{E.2.6} we have $\kappa^\mu\otimes \rho^G$ equivalent to
  $(\dim\kappa_\mu)\otimes\rho^G$.  It follows that for all $f\in C_c(G)$ we have
  \[\|\kappa^\mu(f)\|\le \|(\kappa^\mu\otimes \rho^G)(f)\|=\|((\dim \kappa^\mu)\otimes \rho^G)(f)\|=\|\rho^G(f)\|.\]
\end{proof}

\bigskip

\section{The Renault-Deaconu groupoid}

\bigskip

Let $X$ be a locally compact Hausdorff space and let $T:X\to X$ be a local homeomorphism. Then the Renault-Deaconu groupoid $G(X,T)$ associated to $T$ was described in Example \ref{RD}.
\begin{rmk}
A probability measure $\mu$ on $X=G^{(0)}$ defines a state $\phi_\mu$ on $C^*(G(X,T))$ such that $\ds \phi_\mu(f)=\int_{G^{(0)}}f|_{G^{(0)}}d\mu$ for $f\in C_c(G(X,T))$. It is known that $\phi_\mu$ is a KMS state for the $\RR$-action given by $\alpha_t(f)(\gamma)=e^{itc(\gamma)}f(\gamma)$ at inverse temperature $\beta$ iff $\mu$ is quasi-invariant for $G(X,T)$ with Radon-Nikodym derivative $D_\mu=e^{-\beta c}$, see \cite{KR}. Here $c:G(X,T)\to \ZZ,\; c(x,k,y)=k$.
\end{rmk}

If $\psi:X\to (0,\infty)$ is continuous, then there is a continuous cocycle $D_\psi:G(X,T)\to (0,\infty)$ given by
\[D_\psi(x,m-n,y)=\frac{\psi(x)\psi(Tx)\cdots \psi(T^{m-1}x)}{\psi(y)\psi(Ty)\cdots\psi(T^{n-1}y)}.\]

The transfer operator $\Ll_\psi:C(X)\to C(X)$ is given by
\begin{equation}\label{eq:trop}(\Ll_\psi f)(x)=\sum_{Ty=x}\psi(y)f(y).
\end{equation}
We recall the following result, see \cite{KR} and Proposition 3.4.1 in \cite{ R}.

\begin{prop}If $\mu$ is a probability measure on $X$, then $\mu$ is quasi-invariant for $G(X,T)$ with Radon-Nikodym derivative $\ds D_\mu=\frac{dr^*\mu}{ds^*\mu}$ if and only if  $\Ll_\psi^*\mu=\mu$, where $\Ll^*_\psi$ is the dual operator acting on the space of finite measures on $X$.
\end{prop}

\begin{example}
Assume that $(X,d)$ is a metric space and $T:X\to X$ is a local
homeomorphism such that $\ds \lim_{y\to
  x}\frac{d(Tx,Ty)}{d(x,y)}=\varphi(x)>0$ for all $x\in X$. Let
$\psi(x)=\varphi(x)^{-s}$, where  $s$ is the Hausdorff dimension of
$(X,d)$, Then the Hausdorff measure $\mu$ of $d$ is quasi-invariant
for $G(X,T)$
(\cite{IK13}). 

In particular, for $0<r_j<1, j=1,...,k$ and $X=\{1,2...,k\}^\NN$ with
metric $d$ such that diam$(Z(x_0x_1\cdots x_n))=r_{x_0}r_{x_1}\cdots
r_{x_n}$, the Hausdorff dimension $s$ is the unique solution of the
equation $r_1^s+r_2^s+\cdots+r_k^s=1$ and $\mu$ is given by
$\mu(Z(x_0x_1\cdots x_n))=r_{x_0}^sr_{x_1}^s\cdots r_{x_n}^s$. The
one-sided shift $T:X\to X$ gives $\ds \lim_{y\to
  x}\frac{d(Tx,Ty)}{d(x,y)}=\varphi(x)=\frac{1}{r_{x_0}}$ and  $(\mu, s)$ is such that $\ds \frac{dT^*\mu}{d\mu}=\varphi^s$.
\end{example}

Suppose $G=G(X,T)$ acts on the left on the space $Y$ via $\omega:Y\to X$. Define\[ \tilde{T}:Y\to Y,\;  \tilde{T}(z)=(T(\omega(z)), -1, \omega(z))\cdot z.\]
 Then $\tilde{T}$ is a local homeomorphism such that $\omega\circ\tilde{T}=T\circ \omega$. Moreover, the action groupoid $G\ltimes Y$ is isomorphic to the groupoid $\tilde{G}=G(Y,\tilde{T})$ via the map
\[\Psi:\tilde{G}\to G\ast Y,\; \Psi((z,m-n,y))=(\omega(z),m-n,\omega(y)), y),\]
for $z,y\in Y$, see \cite{IK}.

\begin{cor}\label{cor:qimeas}
In particular, we can construct quasi-invariant measures on $Y$ as quasi-invariant measures on $G\ltimes Y\cong G(Y,\tilde{T})$.
\end{cor}

\begin{example}
Let $E$ be a locally finite directed graph which has no sources. Let
$\ds E^* :=\bigcup_{ k\ge 0} E^k$ be the space of finite paths,
where \[E^k = \{e_1e_2\cdots e_k : e_i \in E^1,\; r(e_{i+1}) =
  s(e_i)\},\] and let $E^\infty$ be the infinite path space with the
topology given by  $Z(\alpha)=\{\alpha x:x\in E^\infty\}$ for
$\alpha\in E^*$.  We assume that $E^\infty$ is a totally disconnected
space, homeomorphic to the Cantor set. On $X=E^\infty$, consider  the
shift $T:X\to X,\;\; T(x)_i=x_{i+1}$ which is a local
homeomorphism. The groupoid $G(X,T)$ is called the graph groupoid and
its $C^*$-algebra is denoted by $C^*(E)$. 

Recall that for $\alpha, \beta\in E^*$ with $s(\alpha)=s(\beta)$, we denote
\[Z(\alpha, \beta)=\{\gamma\in G(X,T)\; |\; \gamma=(\alpha x, |\alpha|-|\beta|,\beta y)\},\]
which are compact open bisections.
The indicator functions $\{1_{Z(v,v)}\;|\; v\in E^0\}$ and $\{1_{Z(e,s(e))}\; |\; e\in E^1\}$ generate $C^*(E)$, see \cite{KPRR} (where the range and source maps are reversed). 

If $G(X,T)$ acts on its unit space $X$ by $(x,k,y)\cdot y=x$, let $\mu$ be the Markov  measure on $X$ determined by    a map
 $p:E\to (0,\infty)$ satisfying $\ds \sum_{r(e)=v}p(e)=1$ for every
 $v\in E^0$ and a map $\mu_0:E^0\to (0,\infty)$ satisfying $\ds
 \sum_{v\in E^0}\mu_0(v)=1$ such that 
\[\mu(Z(e_1e_2\cdots e_n))=\mu_0(r(e_1))p(e_1)p(e_2)\cdots p(e_n).\]
Then $\mu$ is quasi-invariant for $G(X,T)$ and $D(e,1,s(e))=\frac{1}{p(e)}$.

The Koopman representation $\kappa^\mu$ of $G(X,T)$ associated to $\mu$ acts on $L^2(X,\mu)$ by rank $1$ operators, since  $L^2(\omega^{-1}(s(g)), \mu_{s(g)})$ reduces to $\mathbb C$.

We now determine the operators $\kappa^\mu(f)\in \Bb(L^2(X,\mu))$ for the above indicator functions. We have for $\xi\in L^2(X,\mu)$
\[\kappa^\mu(1_{Z(v,v)})\xi(x)=\sum_{r(\gamma)=x}1_{Z(v,v)}(\gamma)\kappa^\mu_\gamma(\xi)(x)=\sum_y 1_{Z(v,v)}((x,k,y))\xi((y,k,x)\cdot x)=\]\[=\sum_y 1_{Z(v,v)}((x,k,y))\xi(y)=\begin{cases}\xi(x)\;\text{if}\; r(x)=v\\0\;\text{if}\; r(x)\neq v,\end{cases}\]
\[\kappa^\mu(1_{Z(e,s(e))})\xi(x)=\frac{1}{\sqrt{p(e)}}\sum_{r(\gamma)=x}1_{Z(e,s(e))}(\gamma)\kappa^\mu_\gamma(\xi)(x)=\]\[=\frac{1}{\sqrt{p(e)}}\sum_y 1_{Z(e,s(e))}((x,k,y))\xi(y)=\begin{cases}\frac{1}{\sqrt{p(e)}}\xi(z)\; \text{if}\; x=ez\\0\;\text{if}\; x\neq ez.\end{cases}\]

Denote by $P_v=\kappa^\mu(1_{Z(v,v)})$ and $S_e=\kappa^\mu(1_{Z(e,s(e))})$. Since $L^2(X,\mu)$ decomposes as $\ds\bigoplus_{v\in E^0}L^2(vX,\mu)$, we note that $P_v$ acts as identity on $L^2(vX,\mu)$ and is $0$ otherwise. It follows that $P_v^*=P_v=P_v^2$ and $\ds\sum_{v\in E^0}P_v=I$. Also, $S_e$ takes $L^2(s(e)X,\mu)$ to $L^2(eX,\mu)$ and
\[S_e^*S_e=P_{s(e)}, \;\;\sum_{r(e)=v}S_eS_e^*=P_v.\] Since $\{P_v,
S_e\}$ satisfy the same relations as $\{1_{Z(v,v)}\}$ and
$\{1_{Z(e,s(e))}\}$ for $v\in E^0, e\in E^1$, it follows that
$C^*(\kappa^\mu)$ is a quotient of $C^*(E)$. Since $\mu_0(v)\ne 0$ for
all $v\in E^0$ and $p(e)\ne 0$ for all $e\in E^1$, it follows that the
partial isometries $S_e$ are all non-zero. Using the same proof as
Theorem 3.7 of \cite{KPR} it follows that $\kappa^\mu$ is faithful
and, thus, $C^*(E)$ is isomorphic with $C^*(\kappa^\mu)$.
\end{example}

\begin{rmk}

By allowing $\mu_0$ and $p$ to take zero values at specific vertices and
edges, one can recover the ideal structure of $C^*(E)$ from the
resulting Koopman representation.

 Assume, for simplicity, that $E$
satisfies condition $(K)$: every vertex $v\in E^0$ either has no
loop based at $v$ or at least two loops based at $v$ (\cite[Section
6]{KPR} where the notation for $r$ and $s$ is reversed compared to ours).
Recall also that a subset $H$ of $E^0$ is called hereditary if
whenever $e\in E^1$ and $s(e)\in H$, then $r(e)\in H$.  The set $H$ is called
saturated if whenever $r(s^{-1}(v))\subset H$, then $v\in H$. It is
known that there is an isomorphism between the lattice of saturated
hereditary subsets of $E^0$ and the lattice of ideals of $C^*(E)$
(\cite[Theorem 6.6]{KPRR}) given via $H\mapsto I(H)$, where
\[
  I(H)=\overline{\text{span}}\{1_{Z(\alpha,\beta)}\,:\,\alpha,\beta\text{
    finite paths with }s(\alpha)=s(\beta)\in H\}.
\]
Let $H$ be a saturated hereditary set and let $\mu_0:E^0\to
[0,\infty)$ and $p:E^1\to [0,\infty)$ be defined such that $\sum_{v\in
  E^0}\mu_0(v)=1$, $\sum_{r(e)=v}p(e)=1$ for all $v\in E^0$,  $\mu_0(v)=0$ for all
$v\in H$ and $p(e)=0$ for all $e\in s^{-1}(H)$. Then $(\mu_0,p)$
defines a quasi-invariant measure $\mu$ on $X$ as above and one can
easily check, using  computations like in the previous example, that $\operatorname{ker}\kappa^{\mu}=I_H$.
\end{rmk}

\begin{example}
  In this example we follow the notation of \cite{IK}: we let
  $W=\{1,\dots,N\}$ for some integer $N\ge 2$, $W^n$ is the set of
  words of length $n$ over the alphabet $W$, and $W^*=\bigcup_{n\ge
    0}W^n$ is the set of finite words over $W$. We let $X=W^\infty$ be
  the set of infinite words (sequences) with elements in $W$ and
  $T:X\to X$ be the shift map:
  $T(x_1x_2x_3\cdots)=(x_2x_3\cdots)$. As in the previous example, the
  topology on $X$ is given by the clopen cylinders $Z(w)=\{w
  x\,:\,x\in X\}$ for all $w\in W^*$. Then $G=G(X,T)$ is the Cuntz
  groupoid (\cite[Section III.2]{R80}) and $C^*(G)$ is isomorphic with
  the Cuntz algebra $ \Oo_N$. Let $(Y,d)$ be a complete metric space and
  let $(F_1,\dots,F_N)$ be an iterated function system on $Y$
  (\cite{Hu}). That is, each $F_i$ is a strict contraction on $Y$. We
  assume further that each $F_i$ is a homeomorphism. There is a unique
  compact invariant set $K$ (\cite[Theorem 3.1.3]{Hu}) such that
  $K=\bigcup_{i=1}^NF_i(K)$. Assume that the iterated function system
  is totally disconnected: $F_i(K)\bigcap F_j(K)=\emptyset$ if $i\ne
  j$. In this case $K$ is a totally disconnected set.

  We recall next
  the construction of a ``fractafold'' bundle $\fS$ on
  which $G$ acts and an invariant measure on $\fS$  as given in
  \cite[Section 3]{IK}. For $w\in W^n$ we  write
  $F_w^{-1}(A)=F_{w_1}\circ \cdots \circ   F_{w_n}^{-1}(A)$ and
  $F_w(A)=F_{w_n}\circ\cdots\circ F_{w_1}(A)$. For $x\in X$ or $x\in
  W^*$ we write $x(n):=x_1\cdots x_n$ and set  
  $\fS_n(x)=F_{x(n)}^{-1}(K)$. Then
  $\fS_n(x)\subset \fS_{n+1}(x)$ and the infinite blow-up of $K$ at $x$ is
  $\fS(x)=\bigcup_{n\ge 0}\fS_n(x)$ endowed with the inductive limit
  topology (see \cite[Section 5.4]{Str} for a short introduction to
  blow-ups). The \emph{fractafold bundle} $\fS$ is defined as the
  increasing union of $\fS_n:=\bigsqcup_{w\in W^*}Z(w)\times
  \fS_n(w)$ endowed with the inductive limit topology. Then $\fS$ is
  a Hausdorff space and the map $\omega:\fS\to
  X$, $\omega(x,t)=x$ is continuous, open and surjective. Under the
  assumption that the iterated function system is totally
  disconnected, $\fS$ is locally compact. The
  groupoid $G(X,T)$ acts on $\fS$ via
  \[
    (x,m-n,y)(y,t)=(x,F_{x(m)}^{-1}(F_{y(n)}(t))).
  \]
  There is a unique invariant probability measure $\mu$ on $K$
  (\cite[Theorem 4.4.1]{Hu}) such that $\displaystyle
  \mu(A)=\frac1{N}\sum_{i=1}^N\mu(F_i^{-1}(A))$ for all Borel subsets
  $A$ of $K$. One can extend $\mu$ to an infinite measure $\mu_x$ on $\fS_x$ via
  $\mu_x(A)=N^n\mu(F_{x(n)}(A))$ if $A\in
  \fS_n(x)$. Consider the measure $\nu$ on $X$ generated by weights
  $\{1/N,\dots,1/N\}$. That is $\nu(Z(w))=(1/N)^n$ for all $w\in W^n$
  and $n\ge 0$. Then there is a unique $G$-\emph{invariant} measure
  $\mu_\infty$ on $\fS$ such that $\mu_\infty(U\times
  A)=\nu(T^n(U))\cdot \mu(F_{w(n)}(A))$ for all $n\ge 0$, $w\in
  W^n$, and $U\times A\subset Z(w)\times \fS_n(w)$ (\cite[Proposition
  3.11]{IK}). Note that the measure $\tm$ in the decomposition of
  $\mu_\infty$ equals $\nu$ and is quasi-invariant for $G$.

  The Koopman representation $\kappa$ of $G$ on $\mu_\infty$ extends
  to a representation of $\Oo_N$ that acts on $L^2(L,\mu_\infty)$ via
  \[
    \kappa(f)\xi(x,t)=\sum_{(x,m-n,y)\in
      G}f(x,m-n,y)\xi(y,F_{y(n)}^{-1}(F_{x(m)}(t)))
  \]
  for all $(x,t)\in L$, $\xi \in L^2(\fS,\mu_\infty)$, and $f\in
  C_c(G)$. In particular, if $S_i=1_{Z(i,\emptyset)}$ are the Cuntz
  isometries generating $C^*(G)\cong \Oo_N$, where
 \[Z(i,\emptyset):=\{(ix,1,x)\,:\,x\in X\},\] then
  \[
    \kappa(S_i)\xi(x,t)=\begin{cases}
                          \xi(T(x),F_i(t))& \text{ if }x\in Z(i)\\
                          0&\text{ otherwise}
                        \end{cases}
 \]
 for all $i=1,\dots, N$. We note that $\kappa(S_i)\ne 0$ for all
 $i=1,\dots N$. To see this, let $\xi \in L^2(\fS,\mu_\infty)$ be
 defined via
 \[
   \xi(x,t)=1_{\fS_0}(x,t)=
   \begin{cases}
     1 & \text{ if }x\in X \text{ and }t\in K\\
     0 &\text{ otherwise.}
   \end{cases}
 \]
 Note that $\xi\in L^2(L,\mu_\infty)$ since $\mu_\infty(\fS_0)=1$. Then
 $\kappa(S_i)(\xi)=\xi_i$ where
 \[
   \xi_i(x,t)=1_{Z(i)\times F_i^{-1}(K)}=
   \begin{cases}
     1 &\text{ if }x\in Z(i)\text{ and }t\in F_i^{-1}(K)\\
     0 &\text{ otherwise},
   \end{cases}
 \] for all $i=1,\dots, N$. Since
 \[\int_{\fS}\xi_i(x,t)\,d\mu_\infty(x,y)=\mu_\infty(Z(i)\times
   F_i^{-1}(K))=\mu(K)=1
   \] by the definition of $\mu_\infty$, we get $\kappa(S_i)\ne 0$. Therefore
   $C^*(\kappa)\cong \Oo_N$.

   One can build other $G$-invariant measures on $\fS$  by considering invariant measures
   for $K$ using non-equal strictly positive weights
   $\{p_1,\dots,p_N\}$ such that $\sum_{i=1}^Np_i=1$. There is a
   unique measure on $K$ that satisfies
   $\mu(A)=\sum_{i=1}^Np_i\mu(F_i^{-1}(A))$ for all Borel subsets $A$
   of $K$. Also, one can define a measure $\nu$ on $X$ based on the
   weights via $\nu(Z(w))=p_{w_1}\cdot \cdots p_{w_n}$ for all $w\in
   W^*$. Then one can prove that the measure $\mu_\infty$ defined as
   above, $\mu_\infty(U\times
  A)=\nu(T^n(U))\cdot \mu(F_{w(n)}(A))$ for all $n\ge 0$, $w\in
  W^n$, and $U\times A\subset Z(w)\times L_n(w)$, is a $G$-invariant
  measure. Under our assumption that $p_i>0$ for all $i=1,\dots,N$, a
  similar analysis proves that $C^*(\kappa)\cong \Oo_N$.
   
\end{example}

\begin{example}
  Consider again the Cuntz groupoid as defined in the previous
  example: $X=\{1,\dots,N\}^\NN$ and $T:X\to X$ is the shift. We show
  that if $Y$ is any left $G$-space and if $\mu$ is any $G$-invariant
  measure on $Y$ with \emph{full support}, then $C^*(\kappa)\cong
  \Oo_N$. This example generalizes easily to the case of finite graphs
  that satisfies the (K)-condition or, equivalently, Cuntz-Krieger
  algebras that satisfy condition (II).

  Let $Y$ be a locally compact Hausdorff left $G$-space with anchor
  map $\omega:Y\to X=\go$ and assume that $\mu$ is a $G$-invariant
  measure on $Y$. Recall that $T$ lifts to a local homeomorphism
  $\tilde{T}:Y\to Y$ defined via
  $\tilde{T}(z)=(T(\omega(z)),-1,\omega(z))\cdot z$ for all $z\in Y$
  and $G\ltimes Y\cong G(Y,\tilde{T})$ (see the discussion before
  Corollary \ref{cor:qimeas}). Therefore there is $\psi:Y\to \RR_+^*$
  such that $\mu$ is invariant for the dual of the transfer operator
  $\Ll_\psi$ defined as in \eqref{eq:trop}. Let
  $S_i=1_{Z(i,\emptyset)}$, $i=1,\dots,N$, be the Cuntz isometries that
  generate $C^*(G)$. Then
  \begin{multline*}
    \kappa(S_i)\xi(z)=\sum_{(\omega(z),m-n,x)\in
      G}S_i(\omega(z),m-n,x)\xi((x,n-m,\omega(z))\cdot z)\\
    \cdots \Delta_{\mu}((x,n-m,\omega(z)),z)^{1/2}\\
    =
    \begin{cases}
      \xi((T(\omega(z)),-1,\omega(z))\cdot
      z)\Delta_\mu((T(\omega(z)),-1,\omega(z)), z)^{1/2}& \text{ if
                                                          }z\in
                                                          \omega^{-1}(Z(i))\\
      0 & \text{otherwise}
    \end{cases}
  \end{multline*}
  which, by the identification of $G\ltimes Y$ with $G(Y,\tilde{T})$
  \begin{align*}
    =&
    \begin{cases}
      \xi(\tilde{T}(z))\Delta_\mu(\tilde{T}(z),-1,z)^{1/2}& \text{ if
                                                            }z\in
                                                            \omega^{-1}(Z(i))\\
      0 & \text{ otherwise}
       \end{cases}\\
      =&      \begin{cases}\xi(\tilde{T}(z))\psi(z)^{-1/2}& \text{ if
                                                            }z\in
                                                            \omega^{-1}(Z(i))\\
      0 & \text{ otherwise.}
    \end{cases}
  \end{align*}
  Since $\psi$ is strictly positive and $\mu$ has full support, it follows that $\kappa(S_i)\ne
  0$ for all $i=1,\dots,N$ and, thus, $C^*(\kappa)\cong \Oo_N$.
\end{example}

\end{document}